\documentclass[10pt]{article}
\usepackage[a4paper,left=3cm,right=3cm,top=2cm,bottom=4cm,bindingoffset=5mm]{geometry}
\usepackage[utf8]{inputenc} 
\usepackage[T1]{fontenc}  
\usepackage{verbatim}   
\usepackage{amsmath}
\usepackage{amsthm}
\usepackage{amsfonts}
\usepackage{mathrsfs}
\usepackage{listings}
\usepackage{framed}
\usepackage{graphicx}
\usepackage{subfigure}
\usepackage{cite}
\usepackage{algorithmic}
\usepackage{algorithm}
\usepackage{color}
\usepackage{pgfplots}
\usepackage{tabularx}
\usepackage{multirow}
\usepackage[colorlinks,citecolor=blue]{hyperref}
\usepackage{ragged2e}

\newtheorem{theorem}{Theorem}[section]
\newtheorem{corollary}{Corollary}[theorem]
\newtheorem{lemma}[theorem]{Lemma}
\newtheorem{assumption}{Assumption}

\newtheorem{remark}{Remark}
\newtheorem{myalgorithm}{Algorithm}

\newcommand\norm[1]{\left\lVert#1\right\rVert}

\newcommand{\R}{\mathbb{R}}
\newcommand{\MO}{\mathcal{M}(\bar{\Omega})}
\newcommand{\CO}{C(\bar{\Omega})}
\newcommand\Lp[1]{L^{#1}(\Omega)}
\newcommand{\Y}{H_0^1(\Omega)\cap \CO}

\newcommand{\alk}{(P_{\alpha,\rho,\mu}^k)}

\newcommand{\Uad}{U_{\mathrm{ad}}}
\newcommand{\Yad}{Y_{\mathrm{ad}}}
\newcommand{\Fad}{F_{\mathrm{ad}}}

\newcommand{\eps}{\varepsilon}

\newcommand{\by}{\bar{y}}
\newcommand{\bu}{\bar{u}}
\newcommand{\bp}{\bar{p}}
\newcommand{\bmu}{\bar{\mu}}

\newcommand{\ya}{{y}^{\alpha}}
\newcommand{\ua}{{u}^{\alpha}}
\newcommand{\pa}{{p}^{\alpha}}
\newcommand{\mua}{{\mu}^{\alpha}}
\newcommand{\lama}{{\lambda}^{\alpha}}

\newcommand{\yak}{{y}^{\alpha_k}}
\newcommand{\uak}{{u}^{\alpha_k}}
\newcommand{\pak}{{p}^{\alpha_k}}
\newcommand{\muak}{{\mu}^{\alpha_k}}

\newcommand{\yan}{y^{\alpha_n}}
\newcommand{\uan}{{u}^{\alpha_n}}

\newcommand{\muan}{{\mu}^{\alpha_n}}

\newcommand\dx{\,\mathrm{d}x}

\newcommand{\la}{\langle}
\newcommand{\ra}{\rangle}

\title{A Joint Tikhonov Regularization and Augmented Lagrange Approach for Ill-posed State Constrained Control Problems with Sparse Controls\thanks{The first author was supported by the German Research Foundation (DFG) within the priority program "Non-smooth and Complementarity-based Distributed Parameter Systems: Simulation and Hierarchical Optimization" (SPP 1962) under grant number Wa 3626/3-1 and the second author was supported under Wa 3626/1-1}}

\author{
	Veronika Karl%
	\thanks{University of Würzburg, Institute of Mathematics,
		Emil-Fischer-Str.\ 30, 97074 Würzburg, Germany; veronika.karl@mathematik.uni-wuerzburg.de}
	\and
	Frank Pörner%
	\thanks{University of Würzburg, Institute of Mathematics,
		Emil-Fischer-Str.\ 30, 97074 Würzburg, Germany; frank.poerner@mathematik.uni-wuerzburg.de}
}

\begin{document}
\maketitle
\begin{abstract}
\noindent We provide a modified augmented Lagrange method coupled with a Tikhonov regularization for solving ill-posed state-constrained elliptic optimal control problems with sparse controls. We consider a linear quadratic optimal control problem without any additional $L^2$ regularization terms. The sparsity is guaranteed by an additional $L^1$ term. Here, the modification of the classical augmented Lagrange method guarantees us uniform boundedness of the multiplier that corresponds to the state constraints. We present a coupling between the regularization parameter introduced by the Tikhonov regularization and the penalty parameter from the augmented Lagrange method, which allows us to prove strong convergence of the controls and their corresponding states. Moreover convergence results proving the weak convergence of the adjoint state and weak*-convergence of the multiplier are provided. Finally, we demonstrate our method in several numerical examples.
\end{abstract}

{\small
\\

\\
}

\section{Introduction}

In this paper we consider a convex optimal control problem of the following form
\begin{gather}
\min\ J(y,u):=\frac{1}{2}||y-y_d||_{L^2(\Omega)}^2 + \beta \|u\|_{L^1(\Omega)}
\label{eq:optcontprob}\tag{$P$}\\
\intertext{subject to}
\begin{alignedat}{2}
Ay &=u &\quad& \text{in }\Omega,\\
y&= 0 && \text{on }\partial\Omega,\\
        y &\leq \psi &&\text{ in } \Omega,\\
        u_a\leq u &\leq u_b && \text{ in } \Omega.
\end{alignedat}
\notag
\end{gather}
We set $j(u) := \|u\|_{L^1(\Omega)}$ for abbreviation. Here $A$ is a linear elliptic operator and $\beta \geq 0$. The main difficulties in this problem are the pointwise state constraints $y(x) \leq \psi(x)$ and the convex but non-differentiable term $\|u\|_{L^1(\Omega)}$. Note that there is no additional $L^2$ regularization term present in \eqref{eq:optcontprob} which makes the problem ill-posed and numerically challenging. To the best of our knowledge, there exists no solution method for this kind of problems in literature.\\

\noindent The motivation for the $L^1$-term in the cost functional is the following. The optimal solution $\bar u$ of \eqref{eq:optcontprob} is sparse, i.e., the control is zero on large parts of the domain if $\beta$ is large enough. This can be used in the optimal placement of controllers, especially in situations where it is not desirable to control the system from the whole domain $\Omega$, see \cite{Stadler2009}. Such sparsity promoting optimal control problems without state constraints have been studied in, e.g. \cite{wachsmuth2013necessary,wachsmuth2011c,wachsmuth2011b} for optimal control of linear partial differential equations and in \cite{casas2012,CasasHerzogWachsmuth2012} for the optimal control of semilinear equations. For sufficient second-order conditions for the state constrained sparsity promoting optimal control problem with a semilinear partial differential equation we refer to \cite{CasasTroeltzsch2014}.\\

\noindent In order to deal with the state constraints we apply an augmented Lagrange method established by the first author in \cite{KarlWachsmuth2017_PREPRINT}. There the optimal control problem
\begin{equation}\label{eq:intro_reg}
\text{Minimize } \frac{1}{2}||y-y_d||_{L^2(\Omega)}^2 + \frac{\alpha}{2}\|u\|_{L^2(\Omega)}^2
\end{equation}
with $\alpha > 0$ subject to an elliptic linear partial differential equation, state constraints and bilateral control constraints had been considered. Under suitable regularity assumptions the existence of Lagrange multipliers can be proven. However in many cases the multiplier $\bar \mu$ has a very low regularity, e.g. $\bar \mu \in C(\Omega)^\ast = \MO$, where $\MO$ denotes the space of regular Borel measures on $\bar \Omega$. This makes the numerical solution of \eqref{eq:optcontprob} very challenging. Although augmented Lagrange method for inequality constraints are well known in finite dimensional spaces, only a few publications considering state constraints in infinite dimensional spaces are available:  In \cite{bergounioux93,bergounioux1993onboundarystate} the state equation is augmented, and in \cite{Ito1990} they deal with finitely many state constraints.\\

\noindent Apart from the augmented Lagrange method there exist some other different approaches to deal with state constraints. We want to mention \cite{lorenz_roesch_2010}, in which a simultaneous Tikhonov and Lavrentiev regularization had been applied for \eqref{eq:optcontprob}. There the motivation was to derive error estimates under a source condition and the assumption that the state constraints are not active for solutions of  \eqref{eq:optcontprob}. Furthermore they assumed that for the lower bound on the control it holds $u_a = 0$. In this paper we do not assume any of the above, which allows us to apply our method to a bigger class of problems.\\

\noindent Our aim is to modify and extend the method presented in \cite{KarlWachsmuth2017_PREPRINT} to obtain a numerical scheme to solve \eqref{eq:optcontprob}. The main idea is the following. We add a Tikhonov regularization term $\frac{\alpha}{2}\|u\|_{L^2(\Omega)}^2$ to \eqref{eq:optcontprob} and apply the augmented Lagrange method. Thus, in every iteration we examine the optimal control problem 
\begin{equation}\label{eq:intro_reg_aug}
\text{Minimize } \frac{1}{2}||y-y_d||_{L^2(\Omega)}^2 +\beta \|u\|_{L^1(\Omega)}+ \frac{\alpha}{2}\|u\|_{L^2(\Omega)}^2+\frac{1}{2\rho}\int_{\Omega}\left(\left(\mu+\rho(y-\psi)\right)_+\right)^2 -\mu^2\ \dx
\end{equation}
subject to an elliptic partial differential equation and bilateral control constraints. Here, again $\alpha>0$ denotes the regularization parameter of the Tikhonov term, while $\rho$ is the penalization parameter of the augmented state constraints. Both variables are coupled in our method. During the algorithm we decrease the regularization parameter $\alpha \to 0$ while increasing the penalization parameter $\rho$. The coupling is described in detail in section \ref{sec:alg_detail}. Since the decrease of $\alpha$ is a classical Tikhonov regularization approach, we aim to achieve strong convergence against the solution of \eqref{eq:optcontprob}.\\

\noindent Denote $\bar u$ the solution of \eqref{eq:optcontprob}, $u^\alpha$ the solution of \eqref{eq:intro_reg} and $u^{\alpha,\rho}$ the solution of \eqref{eq:intro_reg_aug}. Similar to \cite{lorenz_roesch_2010} we split the error into the \textit{Tikhonov error} and the \textit{Lagrange error} in order to show convergence of the algorithm
\[
\|\bar u - u^{\alpha,\rho}\|_{L^2(\Omega)} \leq \underbrace{\|\bar u - u^\alpha\|_{L^2(\Omega)}}_{\text{Tikhonov error}} + \underbrace{\|u^\alpha - u^{\alpha,\rho}\|_{L^2(\Omega)}}_{\text{Lagrange error}}.
\]


\noindent The paper is structured as follows. First, in section \ref{sec:prelim} we recall some preliminary results, then we analyze the Tikhonov regularization in section \ref{sec:regprob}. The augmented Lagrange method will be introduced in section \ref{sec:augmethod}. Similar to \cite{KarlWachsmuth2017_PREPRINT} we only update the multiplier if a certain measure of feasibility and violation of complementarity shows sufficient decrease. In section \ref{sec:convres} we establish convergence of our method, which is the main result of this paper. The convergence is mainly based on an analysis of the Lagrange error. The implementation of our algorithm is described in section \ref{sec:alg_detail} and numerical results are be presented in section \ref{sec:numerics}.

\paragraph{Notation.}
Throughout the article we will use the following notation. The inner product in $L^2(\Omega)$ is denoted by $(\cdot,\cdot)$.
Duality pairings will be denoted by $\langle\cdot,\cdot\rangle$. The dual of $C(\bar{\Omega})$ is denoted by $\mathcal{M}(\bar{\Omega})$,
which is the space of regular Borel measures on $\bar{\Omega}$. Furthermore $c$ is a generic constant which may change from line to line, but is independent from the important variables, e.g. $k$.

\section{Preliminary Results}\label{sec:prelim}

\subsection{Problem Setting}

Let $\Omega\subset \mathbb{R}^N$, $N=\lbrace 1,2,3\rbrace$ be a bounded domain with $C^{0,1}$-boundary $\Gamma$.
Let $Y$ denote the space $Y:= \Y$ and $U := L^2(\Omega)$. We want to solve the following state constrained optimal control problem:
Minimize
\[
J(y,u):=\frac{1}{2}||y-y_d||_{L^2(\Omega)}^2 + \beta \|u\|_{L^1(\Omega)}
\]
over all $(y,u)\in Y\times U$
subject to the elliptic equation
\begin{alignat*}{2}
        (Ay)(x) &= u(x) &\quad& \text{in } \Omega, \\
        y(x) &=0 && \text{on } \Gamma,
\end{alignat*}
and subject to the pointwise state and control constraints
\begin{alignat*}{2}
        y(x)\leq \psi(x) &\quad& \text{in } \Omega,\\
        u_a(x)\leq u(x)\leq u_b(x)&&\text{a.e. in } \Omega.
\end{alignat*}

\noindent In the sequel, we will work with the following set of standing assumptions.
\begin{assumption}\label{ass:standing}
 \begin{enumerate}
  \item The given data satisfy $y_d\in L^2(\Omega)$, $u_a,u_b\in L^\infty(\Omega)$ with $u_a\leq 0 \leq u_b$ and $\psi\in C(\bar \Omega)$.
  \item The differential operator $A$ is given by
  \[
   (Ay)(x) := -\sum_{i,j=1}^N \partial_{x_j}(a_{ij}(x)\partial_{x_i} y(x))
  \]
  with $a_{i,j} \in C^{0,1}(\bar\Omega)$.
  The operator $A$ is assumed to be strongly elliptic, i.e., there is $\delta>0$ such that
  \[
   \sum_{i,j=1}^N a_{ij}(x)\xi_i\xi_j\ge \delta |\xi|^2 \quad \forall \xi\in \R^N, \text{ a.e. on } \Omega.
  \]
 \end{enumerate}
\end{assumption}

\noindent The following theorem is taken from \cite[Theorem 2.1]{casasreyestro08}.
\begin{theorem}\label{theo:exsol-stateeq}
For every $u \in L^2(\Omega)$ there exists a unique weak solution $y \in H_0^1(\Omega) \cap C(\bar \Omega)$ of the state equation and it holds 
\begin{align*}
\norm{y}_{H_0^1(\Omega)} + \norm{y}_{\CO} \leq c \norm{u}_{L^2(\Omega)}.
\end{align*}
with a constant $c>0$ independent of $u$.
\end{theorem}

\noindent With this assumption one can prove the following properties of the control-to-state mapping $S$.

\begin{theorem}\label{thm:S_compact}
The control-to-state mapping $S: L^2(\Omega) \to H_0^1(\Omega) \cap C(\bar \Omega),u\mapsto y$ is a linear, continuous, and compact operator.
\end{theorem}

\begin{proof}
The linearity follows directly by the definition of $S$ and for the compactness we refer \cite[Theorem 2.1]{casasreyestro08}.
\end{proof}

\noindent In the following, we will use the feasible sets with respect to the state and control constraints denoted by
\begin{alignat*}{2}
\Uad&=\lbrace u\in L^\infty(\Omega)\ &&|\ u_a(x)\leq u(x)\leq u_b(x) \text{ a.e. in } \Omega\rbrace,\\
\Yad&=\lbrace y\in C(\bar{\Omega})\ &&|\ y(x)\leq \psi(x)\ \forall x\in \Omega\rbrace.
\end{alignat*}
The feasible set of the optimal control problem is denoted by
\[
\Fad = \lbrace (y,u)\in Y\times L^2(\Omega)\ |\ (y,u)\in \Yad\times \Uad,\ y=Su\rbrace .
\]

\noindent The assumption $u_a \leq 0 \leq u_b$ is not a restriction. Assume that $u_a > 0$ on a subset $\Omega_1 \subseteq \Omega$. Then we can decompose the $L^1$-norm for $u \in \Uad$ as $\|u\|_{L^1(\Omega)} = \|u\|_{L^1(\Omega \setminus \Omega_1)} + \int_{\Omega_1} u$. Hence, on $\Omega_1$ the $L^1$-norm is a linear functional and its treatment does not impose any further difficulties.

\begin{theorem}\label{thm:y_bounded_by_u}
Assume that the feasible set $\Fad$ is non-empty. Then, there exists a unique solution $\bar u$ with associated state $\bar y$ of \eqref{eq:optcontprob}.
\end{theorem}

\begin{proof}
The existence of solutions follows by standard arguments. Due to the assumptions the operator $S$ is linear, continuous, and injective. Hence the problem \eqref{eq:optcontprob} is convex leading to a unique optimal state $\bar y$. By using the injectivity of $S$ we now obtain uniqueness of the optimal control.
\end{proof}

\subsection{Subdifferential of $j$}
In this section we want to recall some basic properties of the subdifferential of the function $j(u) = \|u\|_{L^1(\Omega)}$. Since $j$ is convex and Lipschitz, the generalized gradient (see \cite{Clarke1976}) and the subdifferential in the sense of convex analysis coincide. The subdifferential is defined by
$$\partial j(u) := \left\{ \lambda \in L^\infty(\Omega): \; \int\limits_\Omega \lambda(v-u) \dx \leq \|v\|_{L^1(\Omega)} - \|u\|_{L^1(\Omega)}, \quad \forall v \in L^1(\Omega) \right\}.$$
Since $j$ is a convex function with $\text{dom}(j) = L^1(\Omega)$ the subdifferential is always nonempty. It is easy to compute that $\lambda \in \partial j(u)$ if and only if
$$ \lambda \begin{cases} =+1 & \text{if } u(x) > 0\\=-1 & \text{if } u(x) < 0\\ \in [-1,+1] & \text{if } u(x) = 0  \end{cases}.$$
For more information we refer to the book of Bonnans and Shapiro \cite[Section 2.4.3]{bonnans2000}. We will need the subdifferential to establish derivatives of the objective functional $\frac{1}{2}\|y-y_d\|^2 + \beta \|u\|_{L^1(\Omega)}$ and to obtain optimality conditions.

\subsection{Optimality Conditions}
The existence of Lagrange multpliers cannot be guaranteed without any further regularity assumptions. Throughout this paper will assume that the following Slater condition is satisfied.

\begin{assumption}
\label{ass:slater}
We assume that there exists $\hat{u}\in \Uad$
and $\sigma>0$ such that for $\hat{y} = S\hat{u}$ it holds
\[
\hat{y}(x) \leq \psi(x)-\sigma \quad \forall x\in \Omega.
\]
\end{assumption}
\noindent The choice of Assumption \ref{ass:slater} as regularity condition is motivated as follows. The given inequality in the Slater condition coincides with $\psi-\hat{y}$ lying in the interior of the nonnegative cone of $Y$. The nonnegative cone of $Y=\Y$ equipped with its natural norm $\|\cdot\|_Y := \|\cdot\|_{H_0^1(\Omega)} + \|\cdot\|_{C(\Omega)}$ has nonempty interior - in contrast to $L^p(\Omega),\; p\in[1,\infty)$ equipped with the $L^p$-norm. This implies a possible existence of a Slater point $\hat u$ that satisfies Assumption \ref{ass:slater}. Moreover, since $S$ is linear, Assumption \ref{ass:slater} is equivalent to the linearized Slater condition, which on the other hand implies the more general Zowe-Kurcyusz regularity condition (see \cite[p.332]{troeltzsch2010}). However, since the set of feasible controls may have no interior points (for an example see \cite{troeltzsch2010}), the Zowe-Kurcyusz regularity condition does not imply the linearized Slater condition. Furthermore, one already has to know the solution of the optimal control problem \eqref{eq:optcontprob} to check whether the Zowe-Kurcyusz condition is satisfied. This is not the case for the proposed Slater condition.
\begin{theorem}
\label{theo:ex-adjoint-multiplier}
Let $(\bu,\by)$ be a solution of the problem \eqref{eq:optcontprob}. Furthermore, let Assumption \ref{ass:slater} be fulfilled. Then, there exists an adjoint state $\bp\in W_0^{1,s}(\Omega)$, $s\in[1,N/(N-1))$, a Lagrange multiplier $\bar{\mu}\in \MO$ and a subdifferential $\bar \lambda \in \partial j(\bu)$ such that the following optimality system
\begin{subequations}
\begin{equation}
\left\{\begin{alignedat}{2}
A\by&=\bu  &\quad &\text{in } \Omega,\\
\by&=0 &&\text{on } \Gamma,
\end{alignedat} \right.
\label{eq:kkt_o:2}
\end{equation}
\begin{equation}
\left\{\begin{alignedat}{2}
A^*\bp &= \by-y_d +\bar{\mu} &\quad &\text{in } \Omega,\\
\bp&= 0 &&\text{on } \Gamma,
\end{alignedat}\right.
\label{eq:kkt_o:1}
\end{equation}
\begin{equation}
( \bp + \beta \bar \lambda ,u-\bu) \geq 0\quad\forall u\in U_{ad}, \label{eq:kkt_o:3}
\end{equation}
\begin{equation}
\langle\bar{\mu},\by-\psi\rangle_{\MO,\CO}=0,\quad  \bar{\mu}\geq 0,\label{eq:kkt_o:4}
\end{equation}
\label{eq:kkt_o}
\end{subequations}
is fulfilled. Here, the inequality $\bar{\mu}\geq 0$ means $\langle \bar{\mu},\varphi\rangle_{\MO,\CO}\geq 0$ for all $\varphi\in \CO$ with $\varphi\ge0$.
\end{theorem}
\begin{proof}
The proof can be found in \cite[Theorem 2.5]{CasasTroeltzsch2014}.
\end{proof}

\noindent In the definition \eqref{eq:kkt_o:1} for the optimal adjoint state $\bar p$ we have to solve an elliptic equation with a measure on the right hand side. This problem is well-posed in the following sense.

\begin{theorem}\label{theo:exsol_adjointeq}
Let $\bmu\in\mathcal{M}(\bar{\Omega})$ be a regular Borel measure. Then the adjoint state equation
\begin{equation*}
\begin{alignedat}{2}
A^*{\bp}&=\by-y_d+{\bmu}&\quad& \text{in }\Omega,\\
{\bp}&=0&&\text{on }\Gamma
\end{alignedat}
\end{equation*}
has a unique very weak solution $\bp\in W_0^{1,s}(\Omega)$, $s\in[1,N/(N-1))$, and it holds
\begin{equation}
\norm{\bp}_{W_0^{1,s}(\Omega)}\leq c\left( \norm{\by}_{L^2(\Omega)}+\norm{y_d}_{L^2(\Omega)}+\norm{\bmu}_{\MO}\right).
\end{equation}
\end{theorem}
\begin{proof}
This result is due to \cite[Theorem 4]{casas1986control}.
\end{proof}

\noindent The next theorem shows the relation between the adjoint state and the control. One can see, that if $\beta$ is large, the control will be zero on large parts of $\Omega$. Hence $\bar u$ is sparse.
\begin{lemma}\label{lem:opt_rel_a_a0}
Let $\bar u, \bar p, \bar \lambda, \bar \mu$ satisfy the optimality system. \eqref{eq:kkt_o:2}-\eqref{eq:kkt_o:4}. Then the following relations hold for $\theta > 0$:
\begin{align*}
 \bar u(x) &\begin{cases} = u_a(x) & \text{if} \quad \bar p(x) > \beta\\
= u_b(x) & \text{if} \quad \bar p(x) < -\beta\\
= 0 & \text{if} \quad |p(x)| < \beta\\
\in [u_a(x), u_b(x)] & \text{if} \quad |p(x)| = \beta  \end{cases},\\ 
\bar \lambda(x) &= P_{[-1,+1]}\left( - \frac{1}{\beta} \bar p(x) \right),\\
\bar u(x) &= P_{[u_a(x), u_b(x)]} \big( \bar u(x) - \theta (\bar p(x) + \beta \bar \lambda(x))  \big).
\end{align*}

\noindent From the second formula it follows that $\bar \lambda$ is unique if the multiplier $\bar \mu$ and adjoint state $\bar p$ are unique.
\end{lemma}

\begin{proof}
The proof only uses the optimality \eqref{eq:kkt_o:3} and can be found in \cite[Theorem 3.1]{casas2012}.
\end{proof}

\section{Convergence Analysis of the Regularized Problem }\label{sec:regprob}
Solving the problem \eqref{eq:optcontprob} directly is challenging for mainly two reasons. First, since the multiplier corresponding to the state constraints appears in form of a measure, it is not clear how to deal with the state constraints. For the control constraints many powerful methods are available. Here, we only want to mention the semi-smooth Newton solvers \cite{Hinze2005, Hinze2009b} and the Active-Set methods \cite{Bergounioux1999}. However it is not clear how to implement the state constraints into a direct solver. In \cite{Bergounioux2002,ito2003semi-smooth} Active-Set methods has been used to solve problems where the state constraints have been treated by Moreau-Yosida regularization. In \cite{ito2003semi-smooth} also relations between semi-smooth Newton methods and Active-Set methods have been established that can be used to prove fast local convergence. In this work we want to adapt the approach of a modified augmented Lagrange method that has been proposed by the first author in \cite{KarlWachsmuth2017_PREPRINT} to overcome the lack of the multiplier's regularity. \medskip\\
\noindent The second challenge is the ill-posedness of the original problem \eqref{eq:optcontprob}. There small perturbations of the given data $y_d$ may lead to large errors in the associated optimal controls. To deal with this issue we will use the well-known Tikhonov regularization technique with some positive regularization parameter $\alpha > 0$. The regularized problem is given by
\begin{alignat}{4}
\begin{split}
\text{Minimize } J_\alpha(y,u):&=\frac{1}{2}||y-y_d||_{L^2(\Omega)}^2 + \beta \|u\|_{L^1(\Omega)} +\frac{\alpha}{2}||u||_{L^2(\Omega)}^2\\
s.t. \quad Ay&=u \quad \text{in } \Omega,\\
y&=0\quad \text{on }\partial\Omega,\\
y&\leq \psi,\\
u&\in \Uad.
\end{split}
\tag{$P^{\alpha}$}
\label{eq:RegProb}
\end{alignat}
It is clear that \eqref{eq:RegProb} omits a unique solution $u^\alpha$ with associated state $y^\alpha$. One can expect that $u^\alpha$ converges to the solution of \eqref{eq:optcontprob} as $\alpha \to 0$. Similar results can be found in the literature, e.g. \cite{wachsmuth2011c}.

\begin{lemma}\label{lem:conv_alpha_tech}
Let $\ua$ be the unique solution of \eqref{eq:RegProb} with $\alpha > 0$ with associated state $y^\alpha$. Furthermore let $\bar u$ be the unique solution of \eqref{eq:optcontprob} and $\bar y$ its associated optimal state. Then we have 
\begin{align*}
\|u^\alpha - \bar u\|_{L^2(\Omega)} &\to 0,\\
\frac{1}{\alpha} \|y^\alpha - \bar y\|_{L^2(\Omega)}^2 &\to 0
\end{align*}
as $\alpha \to 0$.
\end{lemma} 

\begin{proof}
We first show that $\|\ua\|_{L^2(\Omega)} \leq \|\bar u\|_{L^2(\Omega)}$ for all $\alpha > 0$. Let $J_0$ denote the cost functional $J_\alpha$ for $\alpha:=0$. We start with
\[
J_0(\ua) + \frac{\alpha}{2}\|\ua\|_{L^2(\Omega)}^2 = J_\alpha(\ua) \leq J_\alpha(\bar u) = J_0(\bar u) + \frac{\alpha}{2}\|\bar u\|_{L^2(\Omega)}^2 \leq J_0(\ua) + \frac{\alpha}{2}\|\bar u\|_{L^2(\Omega)}^2,
\]
where we exploited the optimality of $\ua$ for \eqref{eq:RegProb} and the optimality of $\bu$ for \eqref{eq:optcontprob}.This yields $\|\ua\|_{L^2(\Omega)} \leq \|\bar u\|_{L^2(\Omega)}$. Now we use that the set $\Uad$ is weakly compact and extract a subsequence $u^{\alpha_i} \rightharpoonup u^\ast \in \Uad$. Since the operator $S$ is compact, see Theorem \ref{thm:S_compact}, we obtain strong convergence of the state on the subsequence $y^{\alpha_i} \to y^\ast = S u^\ast$ in $\Y$. Now let $u \in \Uad$ be arbitrary, then
\[
J_0(u^*) = \lim\limits_{i \to \infty} J_0 (u^{\alpha_i}) = \lim\limits_{i \to \infty} J_{\alpha_i}(u^{\alpha_i}) \leq \lim\limits_{i \to \infty} J_{\alpha_i}(u) = J_0(u).
\]
Hence $u^\ast$ is a minimizer of $J_0$. The solution $\bar u$ of \eqref{eq:optcontprob} is unique and since the problems \eqref{eq:optcontprob} and \eqref{eq:RegProb} coincide for $\alpha = 0$ we obtain $\bar u = u^\ast$.  As the norm is weakly lower semicontinuous we get
\[
\limsup\limits_{i \to \infty}\|u^{\alpha_i}\|\leq\|u^\ast\|_{L^2(\Omega)} \leq \liminf\limits_{i \to \infty} \|u^{\alpha_i}\|_{L^2(\Omega)} \leq \limsup\limits_{i \to \infty} \|u^{\alpha_i}\|_{L^2(\Omega)}
\]
which shows $\|u^{\alpha_i}\|_{L^2(\Omega)} \to \|u^\ast\|_{L^2(\Omega)}$. As a well known fact, weak and norm convergence yield strong convergence and hence we have $u^{\alpha_i} \to u^\ast$. 
As the sequence $u^{\alpha_i}$ was arbitrarily chosen we obtain convergence of the whole sequence $\ua \to \bar u$.

We now want to show improved convergence results for the states. Since $S$ is a linear, continuous, and injective operator we know that the functional
\[
J_0(u) = \frac{1}{2} \|Su - y_d\|^2_{L^2(\Omega)}
\]
is strongly convex. In the following let $u$ and $y:=Su$ such that $(u, y) \in \Fad$. Then the following inequality holds for all $t \in [0,1]$
\[
J_0(t \bar u + (1-t) u) \leq t J_0(\bar u) + (1-t) J_0(u) - m \cdot t(1-t) \|\bar u - u\|^2_{L^2(\Omega)}
\]
for some parameter $m> 0$. Now we set $t = \frac{1}{2}$ and use the optimality of $\bar u$ to obtain
\[
J_0(\bar u) \leq \frac{1}{2} J_0(\bar u) + \frac{1}{2} J_0(u) - \frac{m}{2} \|\bar u - u\|_{L^2(\Omega)}^2.
\]
Please note that with $(u, y), (\bar u, \bar y) \in \Fad$ the convex combination is also feasible. Here we set $\bar y := S \bar u$. Furthermore we obtain with the continuity of $S$ that $\|y - \bar y\|_{L^2(\Omega)} \leq c \|u-\bar u\|_{L^2(\Omega)}$ with some constant $c > 0$. Rearranging the inequality above yields the growth condition
\[
J_0(\bar u) + c  \|y-\bar y\|_{L^2(\Omega)}^2 \leq J_0(u).
\]
This growth condition can now be used to established improved convergence results for the states $(y^\alpha)$. Recall that $J_\alpha(u^\alpha) \leq J_\alpha(\bar u)$ and estimate
\begin{align*}
J_0(\bar u) + c \|y^\alpha - \bar y\|_{L^2(\Omega)}^2 &+ \frac{\alpha}{2} \|u^\alpha\|_{L^2(\Omega)}^2 \leq J_0(u^\alpha) + \frac{\alpha}{2} \|u^\alpha\|_{L^2(\Omega)}^2 = J_\alpha (u^\alpha)\\
& \leq J_\alpha (\bar u) = J_0(\bar u) + \frac{\alpha}{2} \|\bar u\|_{L^2(\Omega)}^2.
\end{align*}
This implies
\[
\|y^\alpha - \bar y\|_{L^2(\Omega)}^2 \leq c \cdot \alpha \left(  \|\bar u\|_{L^2(\Omega)}^2 -  \|u^\alpha\|_{L^2(\Omega)}^2    \right).
\]
Using the already established strong convergence $u^\alpha \to \bar u$, we get
\[
\lim\limits_{\alpha \to 0} \frac{1}{\alpha} \|y^\alpha - \bar y\|_{L^2(\Omega)}^2 = 0,
\]
which finishes the proof.
\end{proof}

\subsection{Optimality Conditions}

Let us assume that the Slater condition given in Assumption \ref{ass:slater} is satisfied. Then first order necessary optimality conditions can be established for the regularized problem.

\begin{theorem}
\label{theo:ex-adjoint-multiplier-regprob}
Let $(\ua,\ya)$ be the solution of the problem \eqref{eq:RegProb}. Furthermore, let Assumption \ref{ass:slater} be fulfilled. Then, there exists an adjoint state $\pa\in W^{1,s}(\Omega)$, $s\in [1,N/(N-1))$, a Lagrange multiplier $\mua\in \MO$ and a subdifferential $\lama \in \partial j(\ua)$ such that the following optimality system holds:
\begin{subequations}
\begin{equation}
\left\{\begin{alignedat}{2}
A\ya&=\ua  &\quad &\text{in } \Omega,\\
\ya&=0 &&\text{on } \Gamma,
\end{alignedat}\right.
\label{eq:kktreg_o:1}
\end{equation}
\begin{equation}
\left\{\begin{alignedat}{2}
A^*\pa &= \ya-y_d +\mua &\quad &\text{in } \Omega,\\
\pa&= 0 &&\text{on } \Gamma,
\end{alignedat}\right.
\label{eq:kktreg_o:2}
\end{equation}
\begin{equation}
( \pa + \alpha \ua + \beta \lama ,u-\ua) \geq 0\quad\forall u\in U_{ad}, \label{eq:kktreg_o:3}
\end{equation}
\begin{equation}
\langle\mua,\ya-\psi\rangle_{\MO,\CO}=0,\quad  \mua\geq 0.\label{eq:kktreg_o:4}
\end{equation}
\label{eq:kktreg_o}
\end{subequations}
\end{theorem}
\begin{proof}
The proof can be found in \cite[Theorem 2.5]{CasasTroeltzsch2014}.
\end{proof}

\noindent In the following we collect some results similar to Lemma \ref{lem:opt_rel_a_a0}.
\begin{lemma}\label{lem:opt_rel_a_not0}
Let $u^\alpha, y^\alpha, p^\alpha,\lambda^\alpha, \mu^\alpha$ satisfy the optimality system \eqref{eq:kktreg_o:1}-\eqref{eq:kktreg_o:4}. Then the following relations hold:

\begin{align*}
u^\alpha (x) &= \begin{cases} u_a(x) & \text{if}\quad  \beta - \alpha u_a(x) < p^\alpha(x)\\
\frac{1}{\alpha}( \beta - p^\alpha(x) ) & \text{if}\quad \beta \leq p^\alpha(x) \leq \beta - \alpha u_a(x)\\
0 & \text{if}\quad |p^\alpha(x)| < \beta\\
\frac{1}{\alpha}(-\beta - p^\alpha(x)) & \text{if}\quad - \alpha u_b(x) - \beta \leq p^\alpha(x) \leq - \beta\\
u_b(x) & \text{if}\quad p^\alpha(x) < - \alpha u_b(x) - \beta
\end{cases}\\
\lambda^\alpha (x) &= P_{[-1,1]}\left( - \frac{1}{\beta}(p^\alpha(x) + \alpha u^\alpha (x)) \right)\\
u^\alpha(x) &= P_{[ u_a(x), u_b(x) ]} \left( - \frac{1}{\alpha}( p^\alpha(x) + \beta \lambda^\alpha(x) )  \right)
\end{align*}
\end{lemma}

\begin{proof}
These results can be proven by using a pointwise interpretation of the optimality condition \eqref{eq:kktreg_o:3}.
\end{proof}

\noindent In the subsequent analysis we will need that the multipliers for the problem \eqref{eq:RegProb} are uniformly bounded for all $\alpha \geq 0$. Note that for $\alpha = 0$ the problem \eqref{eq:RegProb} reduces to problem \eqref{eq:optcontprob}. The boundedness of the multiplier can be expected from abstract theory \cite{bonnans2000}, and we make use of the Slater condition to prove it.

\begin{lemma}\label{lem:reg_mult_bounded}
Let $\alpha \geq 0$ and define the set
$$M^\alpha := \{ \mu^\alpha \in \MO: \; (u^\alpha, y^\alpha, p^\alpha, \lambda^\alpha,\mu^\alpha) \text{ satisfy } \eqref{eq:kktreg_o:1}-\eqref{eq:kktreg_o:4} \}.$$
of all multipliers associated with problem \eqref{eq:RegProb}. Then the multipliers are uniformly bounded, i.e. there exists a constant $C > 0$ independent from $\alpha$ such that
$$\|\mu^\alpha \|_{\MO} \leq C, \quad \forall \mu^\alpha \in M^\alpha \quad \forall \alpha \geq 0.$$
\end{lemma}

\begin{proof}
We follow the book of Tröltzsch \cite{troeltzsch2010} and consider our solution mapping $S: L^2(\Omega) \to H_0^1(\Omega) \cap C(\bar \Omega)$. Then the dual operator is a mapping $S^\ast: \MO \to L^2(\Omega)$. Let $\alpha \geq 0$ be given, and $\ua, \ya$ be the solution of \eqref{eq:RegProb} with an associated multiplier $\mu^\alpha$.  We now use the Slater condition from Assumption \ref{ass:slater} and compute for any $f \in C(\bar \Omega)$ with $\|f\|_\infty = 1$:
\begin{align*}
\sigma \left| \int\limits_\Omega f d\mu^\alpha \right| &\leq \sigma \int\limits_\Omega |f| d\mu^\alpha \leq  \int\limits_\Omega \sigma d\mu^\alpha \leq \int\limits_\Omega (\psi - \hat y) d\mu^\alpha\\
&= \underbrace{\langle \mu^\alpha, \psi - y^\alpha \rangle_{\MO, C(\bar \Omega)}}_{=0 \text{ by } \eqref{eq:kktreg_o:4}} + \langle\mu^\alpha, y^\alpha - \hat y \rangle_{\MO, C(\bar \Omega)}\\
&=\langle\mu^\alpha, S(u^\alpha - \hat u) \rangle_{\MO, C(\bar \Omega)}\\
&=\int\limits_\Omega (S^\ast \mu^\alpha) (u^\alpha - \hat u) \dx.
\end{align*}
Now recall that the adjoint equation \eqref{eq:kktreg_o:2} can be rewritten as
$$S^\ast \mu^\alpha = S^\ast(y_d - Su^\alpha) - p^\alpha.$$
Furthermore by assumption $u^\alpha \in \Uad$ and by Theorem \ref{thm:y_bounded_by_u} and \ref{theo:exsol_adjointeq} we obtain that $u^\alpha$, $y^\alpha$ and $p^\alpha$ are uniformly bounded in $L^2(\Omega)$. This now yields

\begin{align*}
\sigma \|\mu^\alpha \|_{\MO} &= \sigma \sup\limits_{f\in C(\bar \Omega), \; \|f\|_\infty = 1} \left| \int\limits_\Omega f d\mu^\alpha \right|\\
&\leq \int\limits_\Omega (S^\ast \mu^\alpha) (u^\alpha - \hat u) \dx\\
&= \int\limits_\Omega (S^\ast(y_d - Su^\alpha) - p^\alpha) (u^\alpha - \hat u) \dx\\
&\leq c\|u^\alpha - \hat u\|_{L^2(\Omega)} (\| y_d - y^\alpha \|_{L^2(\Omega)} + \|p^\alpha\|_{L^2(\Omega)})\\
&\leq c.
\end{align*} 
Dividing the above inequality by $\sigma > 0$ finishes the proof.

\end{proof}

\section{The Augmented Lagrange Method}\label{sec:augmethod}
In the following we want to solve the regularized Problem \eqref{eq:RegProb} for $\alpha \rightarrow 0$. For fixed $\alpha$ we follow  the idea presented in \cite{KarlWachsmuth2017_PREPRINT} and replace the inequality constraint $y\leq\psi$ by an augmented penalization term. In that way we get rid of the measure and work instead with a more regular approximation.
\subsection{The Augmented Lagrange Optimal Control Problem}

First let us introduce the penalty function $P$ which we use to augment the state constraints. Let $\rho>0$ be a given penalty parameter, and let $\mu\in L^2(\Omega)$ with $\mu\ge0$ be a given approximation of the Lagrange multiplier. Now we define

\begin{align}
P(y,\rho,\mu):=\frac{1}{2\rho}\int_{\Omega}\left(\left(\mu+\rho(y-\psi)\right)_+\right)^2 -\mu^2\ \dx.\
\label{def:auglagterm}
\end{align}

\noindent Let now $\rho>0$ and $\mu\in L^2(\Omega)$ be given. Then in each step of the augmented Lagrange method the following
sub-problem has to be solved: Minimize
\begin{equation}
J_\rho^\alpha(y,u,\mu):=\frac{1}{2}||y-y_d||_{L^2(\Omega)}^2+  \beta \|u\|_{L^1(\Omega)}+ \frac{\alpha}{2}||u||_{L^2(\Omega)}^2+P(y,\rho,\mu)
\label{prob_auglag}\tag{$P_{\alpha,\rho,\mu}$}
\end{equation}
with $\alpha > 0$, subject to the state equation and the control constraints
\[
 y = Su, \quad u\in U_{ad}.
\]
A solution of \eqref{prob_auglag} will be denoted by $u_\rho^\alpha$ with associated state $y_\rho^\alpha$ and adjoint state $p_\rho^\alpha$. The next theorem shows that the sub-problem is unique solve-able.

\begin{theorem}[Existence of solutions of the augmented Lagrange sub-problem]
For every $\rho>0$, $\mu\in L^2(\Omega)$ with $\mu\geq 0$ the augmented Lagrange control problem \eqref{prob_auglag} admits a unique solution $u_\rho^\alpha\in U_{ad}$ with associated optimal state $y_\rho^\alpha \in Y$ and adjoint state $p_\rho^\alpha$.
\end{theorem}
\begin{proof}
Since $U_{ad}$ is closed, bounded and convex and $J_\rho^\alpha$ is coercive, weakly lower semi-continuous and strictly convex, problem \eqref{prob_auglag} has a unique solution $u_\rho^\alpha\in U_{ad}$. For more details see \cite{troeltzsch2010} and \cite{reyes2015numerical}.
\end{proof}

\begin{theorem}[First-order necessary optimality conditions]
\label{thm_optcon_aug}
Let $(u_\rho^\alpha,y_\rho^\alpha)$ be the solution of \eqref{prob_auglag}.
Then, there exists a unique adjoint state $p_\rho^\alpha\in H_0^1(\Omega)$ associated with the optimal control $u_\rho^\alpha$ and a subdifferential $\lama_\rho \in \partial j(u_\rho^\alpha)$, satisfying the following system.
\begin{subequations}\label{eq:kkt_auglag}
\begin{equation}
\label{eq:kkt_auglag:1}
\left\{\begin{alignedat}{2}
Ay_\rho^\alpha&=u_\rho^\alpha&\quad &\text{in } \Omega,\\
y_\rho^\alpha&=0 &&\text{on } \Gamma,
\end{alignedat}\right.
\end{equation}
\begin{equation}
\left\{\begin{alignedat}{2}
A^*p_\rho^\alpha &= y_\rho^\alpha-y_d +\mu_\rho^\alpha&\quad &\text{in } \Omega,\\
p_\rho^\alpha&=0&&\text{on } \Gamma,
\end{alignedat}\right.
\label{eq:kkt_auglag:3}
\end{equation}
\begin{equation}
( p_\rho^\alpha + \alpha u_\rho^\alpha + \beta \lama_\rho, u-u_\rho^\alpha) \geq 0\quad \forall u\in U_{ad}, \label{eq:kkt_auglag:2}
\end{equation}
\begin{equation}
\mu_\rho^\alpha:=\left(\mu+\rho(y_\rho^\alpha-\psi)\right)_+. \label{eq:kkt_auglag:4}
\end{equation}
\end{subequations}
\end{theorem}
\begin{proof}
Can be proven by extending the results in \cite[Corollary 1.3, p.\@ 73]{hinze2009optimization}.
\end{proof}
\noindent Further we make an analogue observation like in \cite{KarlWachsmuth2017_PREPRINT}. Boundedness of $\mu_\rho^\alpha$ in the $L^1$-norm is enough to get boundedness of the solution $(u_\rho^\alpha,y_\rho^\alpha,p_\rho^\alpha)$ of \eqref{eq:kkt_auglag}.
\begin{theorem}
 \label{theo:bounded-yup}
Let $\rho>0$ and $\mu\in L^2(\Omega)$ be given. Let $s\in [1,N/(N-1))$ and $\alpha$ be bounded. Then there is a constant $c>0$
 independent of $\alpha$,$\rho$, and $\mu$
 such that for all solutions $(u_\rho^\alpha,y_\rho^\alpha,p_\rho^\alpha,\mu_\rho^\alpha)$ of \eqref{eq:kkt_auglag}
 it holds
 \[
  \|y_\rho^\alpha\|_{H^1(\Omega)}+  \|y_\rho^\alpha\|_{C(\bar\Omega)} +
  \|u_\rho^\alpha\|_{L^2(\Omega)} + \|p_\rho^\alpha\|_{W^{1,s}(\Omega)} \le c ( \|\mu_\rho^\alpha\|_{L^1(\Omega)} + 1).
 \]
\end{theorem}
\begin{proof}
The proof just differs from the one of \cite[Theorem 3.3]{KarlWachsmuth2017_PREPRINT} concerning the additional subdifferential in \eqref{eq:kkt_auglag:3}. Hence, we give just the most important steps here.
Let us test the state equation \eqref{eq:kkt_auglag:1} with $p_\rho^\alpha$ and the adjoint equation \eqref{eq:kkt_auglag:3}  with $y_\rho^\alpha$. This yields
\[
 ( p_\rho^\alpha,u_\rho^\alpha) = (y_\rho^\alpha-y_d, y_\rho^\alpha) + (\mu_\rho^\alpha,y_\rho^\alpha).
\]
Now fix a $u \in \Uad$ and use it in \eqref{eq:kkt_auglag:2}, yielding
\[
 (y_\rho^\alpha-y_d, y_\rho^\alpha) + (\mu_\rho^\alpha,y_\rho^\alpha) \le ( \alpha u_\rho^\alpha, u-u_\rho^\alpha) + (p_\rho^\alpha , u) +(\beta\lambda_\rho^\alpha,u-u_\rho^\alpha).
\]

\noindent By Young's inequality and exploiting $(\lambda_\rho^\alpha, u - u_\rho^\alpha) \leq \|u\|_{L^1(\Omega)} - \|u_\rho^\alpha\|_{L^1(\Omega)}$, we have
\begin{align*}
\frac{1}{2}\norm{y_\rho^\alpha}^2+\frac{\alpha}{2}\norm{u_\rho^\alpha}^2&\leq \frac{1}{2}\|y_d\|_{L^2(\Omega)}^2+ \|\mu_\rho^\alpha\|_{L^1(\Omega)}\|y_\rho^\alpha\|_{C(\bar\Omega)}  + \frac{\alpha}{2} \| u\|_{L^2(\Omega)}^2
 + \|p_\rho^\alpha\|_{L^2(\Omega)} \|u\|_{L^2(\Omega)} \\
 & + \beta \left( \|u\|_{L^1(\Omega)} - \|u_\rho^\alpha\|_{L^1(\Omega)} \right).
\end{align*}

\noindent Let us fix $\bar s\in (1,N/(N-1))$ such that $W^{1,\bar s}(\Omega)$ is continuously embedded in $L^2(\Omega)$. From Theorem \ref{theo:exsol-stateeq} we now get $\|y_\rho^\alpha\|_{H^1(\Omega)}+\|y_\rho^\alpha\|_{C(\bar \Omega)} \leq c \|u_\rho^\alpha\|_{L^2(\Omega)}$ and from Theorem \ref{theo:exsol_adjointeq} we get $\|p_\rho^\alpha\|_{L^2(\Omega)} \leq c \left( \|y_\rho^\alpha\|_{L^2(\Omega)} + \|y_d\|_{L^2(\Omega)} + \|\mu_\rho^\alpha\|_{L^1(\Omega)}  \right)$. Now using the fact that $u_\rho^\alpha$ is bounded in $L^2(\Omega)$ and $u$ is fixed to obtain the result.

\end{proof}

\subsection{The Prototypical Augmented Lagrange Algorithm}\label{sec:auglagmethod}

In the following, let $\alk$ denote the augmented Lagrange sub-problem \eqref{prob_auglag} for
given penalty parameter $\rho:=\rho_k$, multiplier $\mu:=\mu_k$ and regularization parameter $\alpha := \alpha_k$.
We will denote its solution by $(\bar u_k, \bar y_k)$ with adjoint state $ \bar p_k$
and updated multiplier $\bmu_k$, which is given by \eqref{eq:kkt_auglag:4}.
\begin{myalgorithm}\label{alg_1}
Let $\alpha_1 > 0$, $\rho_1>0$ and $\mu_1\in L^2(\Omega)$ be given with $\mu_1\ge0$. Choose $\theta>1$.
\begin{enumerate}
 \item Solve $\alk$ and obtain $(\bar u_k, \bar y_k, \bar p_k)$.
 \item Set $\bar\mu_k :=(\mu_k+\rho_k(\by_k-\psi))_+$.
 \item If the step is successful set $\mu_{k+1}:=\mu_k$, $\rho_{k+1}:=\rho_k$ and choose $0 < \alpha_{k+1}$ such that $\alpha_{k+1} < \alpha_k$.
 \item Otherwise set $\mu_{k+1}:=\bar \mu_k$ and $\alpha_{k+1} := \alpha_k$, increase penalty parameter $\rho_{k+1}:=\theta \rho_k$.
 \item If the stopping criterion is not satisfied set $k:=k+1$ and go to step 1.
\end{enumerate}
\end{myalgorithm}
\noindent Please note that we only decrease the regularization parameter $\alpha_k$ if the algorithm produces a successful step.
Let us restate the system $\alk$ that is solved by $(\bar u_k, \bar y_k, \bar p_k, \bar \mu_k)$:
\begin{subequations}\label{eq_kkt_alg}
\begin{equation}
\left\{\begin{alignedat}{2}
A \bar y_k&=\bar u_k&\quad &\text{in } \Omega,\\
\bar y_k&=0 &&\text{on } \Gamma,
\end{alignedat}\right.
\label{eq_kkt_alg_1}
\end{equation}
\begin{equation}
\left\{\begin{alignedat}{2}
A^*\bar p_k &= \bar y_k-y_d +\bar \mu_k&\quad &\text{in } \Omega,\\
\bar p_k&=0&&\text{on } \Gamma,
\end{alignedat}\right.
\label{eq_kkt_alg_2}
\end{equation}
\begin{equation}\label{eq_kkt_alg_3_1}
\bar u_k\in \Uad,
\end{equation}
\begin{equation}
\label{eq_kkt_alg_3}
( \bar p_k + \alpha_k \bar u_k + \beta \bar \lambda_k, u- \bar u_k) \geq 0\quad \forall u\in \Uad,
\end{equation}
\begin{equation}
\bar \mu_k:=\left(\mu_k+\rho_k(\bar y_k-\psi)\right)_+.
\end{equation}
\end{subequations}

\subsection{The Multiplier Update Rule}
\label{sec:rule}

We start this section with a technical estimate, which will be useful in the subsequent analysis.
\begin{lemma}\label{lemma:erroryubounded1}
Let $\alpha_k>0$ be given and let $(\uak,\yak,\pak,\muak)$ be the solution of \eqref{eq:kktreg_o} and let $(\bar u_k,\bar y_k,\bar p_k,\bar \mu_k)$ solve \eqref{eq_kkt_alg}.
Then  it holds
\begin{align}
\norm{\yak-\bar y_k}_{\Lp{2}}^2+\alpha_k\norm{\uak- \bar u_k}_{\Lp{2}}^2\leq (\bar \mu_k, \psi- \bar y_k) +\la \mu^{\alpha_k}, \by_k -\psi\ra.
\label{eq:error_yu1}
\end{align}
\end{lemma}
\begin{proof}
Using \eqref{eq:kktreg_o:3} and \eqref{eq_kkt_alg_3}, we obtain
\begin{align}
\begin{split} 
0 & \leq (\pak - \bp_k + \alpha_k ( \uak - \bu_k) + \beta ( \lambda^{\alpha_k} - \bar \lambda_k ), \bar u_k - \uak)\\
& = ( S^\ast (S \uak - S  \bu_k),  \bu_k - \uak ) - \alpha_k ( \bu_k-\uak,  \bu_k - \uak)\\
& \quad + (S^\ast(   \muak - \bmu_k),  \bu_k - \uak) + \beta ( \lambda^{\alpha_k} - \bar \lambda_k , \bar u_k - \uak)
\label{eq:est1}
\end{split}
\end{align}
Now we use that the subdifferential is a monotone operator, which yields $( \lambda^{\alpha_k} - \bar \lambda_k , \bar u_k - \uak) \leq 0$. Note that $\lambda^{\alpha_k} \in \partial j(\uak)$ and $\bar \lambda_k \in \partial j(\bu_k)$. This yields
$$\|\yak -  \by_k\|^2  + \alpha_k \| \bu_k - \uak\|^2 \leq  (  \bmu_k - \muak, \yak - \by_k)$$
The term on the right-hand side of equation \eqref{eq:est1} can be split into two parts:
\begin{align}
\begin{split}
(\bmu_k,\yak-\by_k) &= (\bmu_k,\yak-\psi)+(\bmu_k,\psi-\by_k) \leq (\bmu_k,\psi-\by_k)
\end{split}
\label{eq:mu_term1}
\end{align}
and
\begin{align}
\begin{split}
-\la{\muak_k},\yak-\by_k\ra &= -\la{\muak},\yak-\psi\ra-\la{\muak},\psi-\by_k\ra=\la{\muak},\by_k-\psi\ra.
\end{split}
\label{eq:mu_term2}
\end{align}
Here, we used the complementarity relation \eqref{eq:kktreg_o:4} as well as $\yak\le\psi$ and $\bmu_k\ge0$.
Putting the inequalities \eqref{eq:est1}, \eqref{eq:mu_term1}, and \eqref{eq:mu_term2} together, we get
\begin{align*}
\norm{\yak-\by_k}_{\Lp{2}}^2+\alpha_k \norm{\uak-\bu_k}_{\Lp{2}}^2 \leq (\bmu_k, \psi-\by_k) +\la \muak, y_k -\psi\ra.
\end{align*}
which is the claim.
\end{proof}

\noindent The following result motivates the update rule.
\begin{lemma}\label{lemma:error_yu}
Let $(\uak,\yak,\pak,\muak)$ and $(\bu_k,\by_k,\bp_k,\bmu_k)$ be given as in Lemma \ref{lemma:erroryubounded1}. Then it holds
\begin{align}
\frac{1}{\alpha_k} \norm{\yak-\by_k}_{\Lp{2}}^2+ \norm{\uak-\bu_k}_{\Lp{2}}^2\leq \frac{c}{\alpha_k} \big( \norm{(\by_k-\psi)_+}_{C(\bar{\Omega})} + |( \bmu_k,\psi-\by_k)| \big).
\label{eq:error_yu}
\end{align}
\end{lemma}
\begin{proof}
From Lemma \ref{lemma:erroryubounded1} we conclude using the estimate
\[
\la\bar{\muak},\by_k-\psi\ra\leq \norm{\muak}_{\mathcal{M}(\Omega)}\norm{(\by_k-\psi)_+}_{C(\bar{\Omega})}.
\]
The result now follows using the uniform boundedness of $\muak$, see Lemma \ref{lem:reg_mult_bounded}.
\end{proof}
\noindent This result shows that the iterates $(\bar u_k, \bar y_k)$ will converge to the solution of the regularized problem for fixed $\alpha_k$ if the
quantity
\[
\frac{1}{\alpha_k} \big( \norm{(\bar y_k-\psi)_+}_{C(\bar{\Omega})} + |(\bmu_k,\psi-\by_k)| \big)
\]
tends to zero for $k\to\infty$. To construct our update rule we follow the idea presented in \cite{KarlWachsmuth2017_PREPRINT} and define a step of Algorithm \ref{alg_1} to be successful if the condition
\[
 \frac{1}{\alpha_{k}} \big( \norm{(\by_k-\psi)_+}_{C(\bar{\Omega})} + |( \bmu_k,\psi-\by_k)| \big)
 \le \frac{\tau}{\alpha_n}  \left( \norm{(\by_{n}-\psi)_+}_{C(\bar{\Omega})} + |( \bmu_{n},\psi-\by_{n})| \right)
\]
is satisfied with $\tau\in (0,1)$. Here, we denoted by step $n$, $n<k$, the previous successful step. In \cite{KarlWachsmuth2017_PREPRINT} this quantity was also used as a stopping criterion. However this is not possible here, as we proceed to let $\alpha$ go to $0$. Instead we will check the first order optimality conditions for problem \eqref{eq:optcontprob} as a stopping criterion. This will be described in detail in section \ref{sec:alg_detail}.

\subsection{The Augmented Lagrange Algorithm in Detail}\label{sec:alg}

Let us now formulate the algorithm based on the update rule established in the previous section.

\begin{myalgorithm} \label{alg_detail}
Let $\alpha_1 > 0$, $\rho_1>0$ and $\mu_1\in L^2(\Omega)$ be given with $\mu_1\ge0$. Choose $\theta>1, 0 < \omega < 1$, $\tau\in(0,1)$ and $R_0^+>>1$. Set $k:=1$ and $n:=1$.
\begin{enumerate}
 \item Solve $\alk$ and obtain $(\bu_k,\by_k,\bp_k)$.
 \item Set $\bar\mu_k :=(\mu_k+\rho_k(\by_k-\psi))_+$.
 \item Compute $R_k:= \frac{1}{\alpha_k} \big( \norm{(\by_k-\psi)_+}_{C(\bar{\Omega})} + |( \bmu_k,\psi-\by_k)| \big)$.
 \item If $R_k\le \tau R^+_{n-1}$ then the step $k$ is successful, set 
 $$\begin{cases}\alpha_{k+1} := \omega \alpha_k\\
 \mu_{k+1}:=\bmu_k\\
 \rho_{k+1}:=\rho_k \end{cases}$$ 
and define $(u_n^+,y_n^+,p_n^+):=(\bu_k,\by_k,\bp_k)$, as well as $\mu_n^+:=\mu_{k+1}$ and $R_n^+:=R_k$.
 Set $n:=n+1$.
 \item Otherwise if the step $k$ is not successful, set $\mu_{k+1}:=\mu_k$ and $\alpha_{k+1} := \alpha_k$, and increase the penalty parameter $\rho_{k+1}:=\theta \rho_k$.
 \item If a stopping criterion is satisfied stop, otherwise set $k:=k+1$ and go to step 1.
\end{enumerate}
\end{myalgorithm}

Again, please note that the regularization parameter $\alpha_k$ is only decreased when the algorithm produces a successful step. We will take advantage of this in the subsequent analysis.

\subsection{Infinitely many Successful Steps}
The main aim of this section is to prove that the proposed algorithm produces infinitely many successful steps. In order to prove this we consider the augmented Lagrange KKT system of the minimization problem
$$\text{Minimize }J_\rho^\alpha(y,u,\mu) =\frac{1}{2}||y-y_d||_{L^2(\Omega)}^2+  \beta \|u\|_{L^1(\Omega)}+ \frac{\alpha}{2}||u||_{L^2(\Omega)}^2+P(y,\rho,\mu)$$
subject to $y = Su$ and $u \in \Uad$. We fix the multiplier approximation $\mu$, the regularization parameter $\alpha$ and let the penalization parameter $\rho$ tend to infinity. As mentioned in \cite{KarlWachsmuth2017_PREPRINT} the problem reduces to a penalty method with additional shift parameter $\mu$. The only difference to the approach in \cite{KarlWachsmuth2017_PREPRINT} is, that we have an additional $L^1$-term in the objective functional. However, taking a closer look at \cite[Lemma 3.6]{KarlWachsmuth2017_PREPRINT} reveals that it also holds for an additional $L^1$-term. This yields the following Lemma.
\begin{lemma}
Let $\mu\in L^2(\Omega)$ with $\mu\ge0$ and $\alpha > 0$ be given. Let $(u_\alpha^{\rho},y_\alpha^{\rho},p_\alpha^{\rho})$ be solutions of \eqref{prob_auglag} with $\rho > 0$ and $(\ua,\ya)$ be the solution of \eqref{eq:RegProb}. Then it holds  $u_\rho^\alpha \rightarrow \ua$ in $L^2(\Omega)$ and $y_\rho^\alpha \to \ya$ in $H_0^1(\Omega) \cap C(\bar \Omega)$ for $\rho\to\infty$.
\label{lemma:convrhounbounded}
\end{lemma}

\noindent With a similar argument we can establish the next lemma. Again the proof can be found in \cite[Lemma 3.7]{KarlWachsmuth2017_PREPRINT}.
\begin{lemma}
\label{lemma:convmuinL2}
Under the same assumptions as in Lemma \ref{lemma:convrhounbounded}, it holds
\begin{align*}
\lim_{\rho \to \infty}(\mu_\rho^\alpha,\psi-y_\rho^\alpha)=0.
\end{align*}
\end{lemma}

\noindent If we now combine these two results we can show that our algorithm produces infinitely many successful steps. This will be crucial in the convergence analysis in the next section.
\begin{lemma}
The augmented Lagrange algorithm makes infinitely many successful steps.
\label{lemma:step3ainfinitely}
\end{lemma}

\begin{proof}
We assume that the algorithm produces only finitely many successful steps. Then there is an index $m$ such that all steps $k > m$ are not successful. Due to the definition of the algorithm we obtain $\bmu_k = \bmu_m$ for all $k> m$ and $R_k > \tau R_m > 0$ as well as $\rho_k \to \infty$. This now yields a contradiction as with Lemma \ref{lemma:convrhounbounded} and \ref{lemma:convmuinL2} we obtain
$$0 < \lim\limits_{k \to \infty} R_k = \lim\limits_{k \to \infty} \frac{1}{\alpha_k} \big( \norm{(\by_k-\psi)_+}_{C(\bar{\Omega})} + |( \bmu_k,\psi-\by_k)| \big) = 0$$
Please note that $\alpha_k$ is constant for $k > m$ since its value is only decreased in a successful step.
\end{proof}

\section{Convergence Results}\label{sec:convres}
In this section we want to show convergence of Algorithm \ref{alg_detail}. Let us recall that the sequence $(u_n^+,y_n^+,p_n^+)$ denotes the solution of the $n$-th successful iteration of Algorithm \ref{alg_detail} with $\mu_n^+$ being the corresponding approximation of the Lagrange multiplier.
We start with proving $L^1$-boundedness of the Lagrange multipliers $\mu_n^+$, which is accomplished in Lemma \ref{lemma:L1boundedness} below. To prove this result we need an auxiliary estimation first.
\begin{lemma}
\label{lemma:scaprod}
Let $y_n^+,\mu_n^+$ be given as defined in Algorithm \ref{alg_detail}. Then it holds
\begin{align}
\frac{1}{\alpha_n}|( \mu_n^+,\psi-y_n^+)|\leq \frac{\tau^{n-1}}{\alpha_1}\left(\norm{(y_1^+-\psi)_+}_{C(\bar{\Omega})}+\norm{\mu_1^+}_{L^2(\Omega)}\norm{(\psi-y_1^+)_+}_{L^2(\Omega)}\right).
\label{eq:est_scaprod}
\end{align}
\end{lemma}
\begin{proof}
Using the definition for a successful step we obtain:
 \begin{align*}
 \frac{1}{\alpha_n} |( {\mu_n^+},\psi-y_n^+)|&  
 \leq \frac{\tau}{\alpha_{n-1}}\left( \norm{(y_{n-1}^+-\psi)_+}_{C(\bar{\Omega})} + |({\mu}_{n-1}^+,\psi-y_{n-1}^+)|\right) - \frac{1}{\alpha_n} \norm{(y_{n}^+-\psi)_+}_{\CO} \\
 & \leq \frac{\tau}{\alpha_{n-1}} \norm{(y_{n-1}^+-\psi)_+}_{C(\bar{\Omega})} + \tau \left( \frac{1}{\alpha_{n-1}}|({\mu}_{n-1}^+,\psi-y_{n-1}^+)|  \right)\\
 &\leq \frac{\tau}{\alpha_{n-1}} \norm{(y_{n-1}^+-\psi)_+}_{C(\bar{\Omega})}\\
 &\quad + \tau \bigg(  \frac{\tau}{\alpha_{n-2}}\left( \norm{(y_{n-2}^+-\psi)_+}_{C(\bar{\Omega})} + |({\mu}_{n-2}^+,\psi-y_{n-2}^+)|\right)\\
 &\quad- \frac{1}{\alpha_{n-1}} \norm{(y_{n-1}^+-\psi)_+}_{\CO}  \bigg)\\
 &\leq  \frac{\tau^2 }{\alpha_{n-2}}\left( \norm{(y_{n-2}^+-\psi)_+}_{C(\bar{\Omega})} + |({\mu}_{n-2}^+,\psi-y_{n-2}^+)|\right)
 \end{align*}
 
 The rest now follows by induction and a standard estimate.
 
\end{proof}

\noindent We want to point out that the right hand side of \eqref{eq:est_scaprod} goes to $0$ as $n \to \infty$. This will be crucial in the following convergence analysis and is a result of our update rule. Let us now show the $L^1$-boundedness of the Lagrange multipliers $(\mu_n^+)$.
\begin{lemma}[Boundedness of the Lagrange multiplier]
Let Assumption \ref{ass:slater} be fulfilled.
Then Algorithm \ref{alg_detail} generates an infinite sequence of bounded iterates, i.e.,
there is a constant $C>0$ such that for all $n$ it holds
\[\norm{y_n^+}_{H^1(\Omega)}+\norm{y_n^+}_{\CO}+ \norm{u_n^+}_{L^2(\Omega)}+\norm{p_n^+}_{W^{1,s}(\Omega)}+ \norm{\mu_n^+}_{L^1(\Omega)}\leq C.\]
\label{lemma:L1boundedness}
\end{lemma}

\begin{proof}
Let $(\hat{u},\hat{y})$ be the Slater point given by Assumption \ref{ass:slater}, i.e., there exists $\sigma>0$, such that $\hat{y}+\sigma\leq \psi$. Then we can estimate
\begin{align*}
\sigma||\mu_{n}^+||_{L^1(\Omega)}&=\int_{\Omega}\sigma\mu_{n}^+\dx\leq \int_{\Omega}\mu_{n}^+(\psi-\hat{y})\dx=\int_{\Omega}\mu_{n}^+(\psi-y_n^+ +y_n^+-\hat{y})\dx\\
&=\underbrace{\int_{\Omega}\mu_{n}^+(\psi-y_n^+)}_{(I)}\dx+\underbrace{\int_{\Omega}\mu_{n}^+(y_n^+-\hat{y})\dx}_{(II)}.
\end{align*}
The first part $(I)$ can be estimated with Lemma \ref{lemma:scaprod} yielding
\begin{align}
\begin{split}
(I) \le  |( \mu_n^+,\psi-y_n^+)|&\leq \frac{\alpha_n}{\alpha_{n-1}} \tau^{n-1}\left(\norm{(y_1^+-\psi)_+}_{C(\bar{\Omega})}+\norm{\mu_1^+}_{L^2(\Omega)}\norm{(\psi-y_1^+)_+}_{L^2(\Omega)}\right)\\
&\leq c \tau^{n-1} .
\end{split}
\label{proof:mu_bounded:eq2}
\end{align}
Please note that we used the monotonicity of $(\alpha_n)_n$. Before we estimate part $(II)$ recall that we have the inequality
$$(\lambda_n^+, u - u_n^+) \leq \|u\|_{L^1(\Omega)} - \|u_n^+\|_{L^1(\Omega)}$$
for every $u \in L^1(\Omega)$. By definition we obtain that $u \in \Uad$ implies $u \in L^\infty(\Omega)$. Now the second part $(II)$ can be estimated using Young's Inequality as follows
\begin{align}
\begin{split}
\int_\Omega&\mu_{n}^+(y_n^+-\hat{y})\dx =\la A^*p_n^+-(y_n^+-y_d),y_n^+-\hat{y}\ra \\
&=\la p_n^+,A(y_n^+-\hat{y})\ra-( y_n^+-y_d,y_n^+-\hat{y})\\
&=( p_n^+,u_n^+-\hat{u})-( y_n^+-y_d,y_n^+-\hat{y}) \\
&\leq - ( \alpha u_n^+,u_n^+-\hat{u})-( y_n^+-y_d,y_n^+-\hat{y}) - \beta (\lambda_n^+, u_n^+ - \hat u)\\
&\leq  ( \alpha u_n^+,\hat{u}-u_n^+)+( y_n^+-y_d,\hat{y}-y_n^+) + \beta( \|\hat u\|_{L^1(\Omega)} - \|u_n^+\|_{L^1(\Omega)} )\\
&= \alpha(u_n^+-\hat{u},\hat{u}-u_n^+)+\alpha(\hat{u},\hat{u}-u_n^+)+(y_n^+-\hat{y},\hat{y}-y_n^+)+(\hat{y}-y_d,\hat{y}-y_n^+)\\
&\quad + \beta( \|\hat u\|_{L^1(\Omega)} - \|u_n^+\|_{L^1(\Omega)} )\\
&\leq -\frac{\alpha}{2}\norm{\hat{u}-u_n^+}^2_{L^2(\Omega)} -\frac{1}{2}\norm{\hat{y}-y_n^+}^2_{L^2(\Omega)}+\frac{\alpha}{2}\norm{\hat{u}}^2_{L^2(\Omega)} +\frac{1}{2}\norm{\hat{y}-y_d}^2_{L^2(\Omega)}\\
&\quad +  \beta( \|\hat u\|_{L^1(\Omega)} - \|u_n^+\|_{L^1(\Omega)} ).
\end{split}
\label{proof:mu_bounded:eq1}
\end{align}
Putting \eqref{proof:mu_bounded:eq2} and \eqref{proof:mu_bounded:eq1} together yields
\begin{align*}
\begin{split}
\norm{\mu_n^+}_{L^1(\Omega)}&+\frac{\alpha}{2}\norm{\hat{u}-u_n^+}^2_{L^2(\Omega)}  +\frac{1}{2}\norm{\hat{y}-y_n^+}^2_{L^2(\Omega)} \\
&\leq \frac{\tau^{n-1}}{\sigma}C +\frac{\alpha}{2}\norm{\hat{u}}_{L^2(\Omega)}^2+\beta\|\hat u\|_{L^1(\Omega)}+\frac{1}{2}\norm{\hat{y}-y_d}_{L^2(\Omega)}^2.
\end{split}
\end{align*}
Since $\tau\in(0,1)$ by assumption, the right-hand side is bounded. Consequently we get boundedness of $(u_n^+)$ in $L^2(\Omega)$
and boundedness of $(\mu_n^+)$ in $L^1(\Omega)$.
By the regularity result Theorem \ref{theo:exsol-stateeq},
the sequence $(y_n^+)$ is uniformly bounded in $\Y$.
Boundedness of $(p_n^+)$ follows directly from Theorem \ref{theo:bounded-yup}.
\end{proof}

\begin{theorem}[Convergence of solutions]
\label{theo:convergence}
As $n\rightarrow\infty $ we have for the sequence $( u_n^+,y_n^+ )$ generated by Algorithm \ref{alg_detail}
\begin{align*}
(u_n^+,y_n^+)&\rightarrow (\bar u, \bar y),\ in\ \Lp{2} \times (\Y).
\end{align*}
where $\bar u$ denotes the unique minimum norm solution of problem \eqref{eq:optcontprob}.
\end{theorem}
\begin{proof}
Since the algorithm yields an infinite number of successful steps (Lemma \ref{lemma:step3ainfinitely}) we get
\begin{align}
\lim_{n\rightarrow\infty}R_n^+=\lim_{n\rightarrow\infty}&\frac{1}{\alpha_n}\left(\norm{(y_n^+-\psi)_+}_{\CO}+ |(\mu_n^+,\psi-y_n^+)|\right) = 0.
\label{eq:lim-Rk1}
\end{align}
with $\alpha_n\rightarrow 0$. Let $(u^{\alpha_n},y^{\alpha_n},p^{\alpha_n},\mu^{\alpha_n})$ be a solution of \eqref{eq:kktreg_o} for $\alpha:=\alpha_n$ then we obtain from Lemma \ref{lemma:erroryubounded1} the following inequality
\begin{align*}
\begin{split}
\frac{1}{\alpha_n}\norm{\yan-y_n^+}_{\Lp{2}}^2+\norm{\uan-u_n^+}_{\Lp{2}}^2  &\leq \frac{1}{\alpha_n}\left(\la \muan,y_n^+-\psi\ra +|(\mu_n^+,\psi-y_n^+)|\right)\\
&\leq\frac{1}{\alpha_n}\left(  \norm{\muan}_{\MO}\norm{(y_n^+-\psi)_+}_{\CO}+|(\mu_n^+,\psi-y_n^+)|\right)\\
&\leq\frac{c}{\alpha_n}\left( \norm{(y_n^+-\psi)_+}_{\CO}+|(\mu_n^+,\psi-y_n^+)|\right).
\end{split}
\end{align*}
Note that in the last step we used Lemma \ref{lem:reg_mult_bounded}. With \eqref{eq:lim-Rk1} from above, we conclude
\begin{equation}\label{eq:tech_3}
\lim_{n\rightarrow\infty}\frac{1}{\alpha_n}\norm{\yan-y_n^+}^2_{\Lp{2}}+\norm{\uan-u_n^+}^2_{\Lp{2}}= 0.
\end{equation}
We now split the error as described in the introduction
\[
\|u_n^+ - \bu\|_{L^2(\Omega)} \leq \underbrace{\|u_n^+ - \uan\|_{L^2(\Omega)}}_{\text{Lagrange error (I)}}  + \underbrace{\|\uan - \bu\|_{L^2(\Omega)}}_{\text{Tikhonov error (II)}} .
\]
Using \eqref{eq:tech_3} we obtain $(I) \to 0$.  Now we use the fact that our algorithm creates infinitely many successful steps, which gives $\alpha_n \to 0$ as $n \to \infty$. We therefore conclude that  $ \uan \to \bu$, see Lemma \ref{lem:conv_alpha_tech}. Hence hence $(II) \to 0$. So in total we obtain $u_n^+ \to \bar u$ in $L^2(\Omega)$. Convergence of $y_n^*\rightarrow\by$ follows from Theorem \ref{theo:exsol-stateeq} which finishes the proof.
\end{proof}

\begin{corollary}
For the sequence $(y_n^+)$ generated by Algorithm \ref{alg_detail} we obtain
$$\frac{1}{\alpha_n} \|y_n^+ - \bar y\|_{L^2(\Omega)} ^2 \to 0$$
which is similar to the results obtained for a Tikhonov regularization without state constraints, see \cite{wachsmuth2011c} and Lemma \ref{lem:conv_alpha_tech}.
\end{corollary}
\begin{proof}
We split the error to obtain with some $c > 0$ independent from $n$:
\[
\frac{1}{\alpha_n} \|y_n^+ - \bar y\|_{L^2(\Omega)} ^2 \leq \frac{c}{\alpha_n} \|y_n^+ -  \yan\|_{L^2(\Omega)} ^2 + \frac{c}{\alpha_n} \|\yan - \bar y\|_{L^2(\Omega)} ^2.
\]
The result is now an immediate consequence of \eqref{eq:tech_3} and Lemma \ref{lem:conv_alpha_tech}.
\end{proof}

\noindent Let us assume that the adjoint state $\bp$ and the multiplier corresponding to the state constraint $\bmu$ are unique then following \cite[Theorem 3.12, Corollary 3.12]{KarlWachsmuth2017_PREPRINT} we get the following convergence result.
\begin{theorem}\label{theo:convmult}
Let $(\bu,\by,\bp,\bu)$ satisfy the KKT-system \eqref{eq:kkt_o}. Let us assume that $(\bp,\bmu)$ are uniquely given. Then it holds
\begin{align*}
p_n^+ &\rightharpoonup \bp \qquad\text{in} \ L^2(\Omega),\\
\mu_n^+ &\overset{*}\rightharpoonup \bmu \qquad\text{in}\  \MO.
\end{align*}
\end{theorem}

\section{Numerical Method in Detail}\label{sec:alg_detail}
All optimal control problems have been solved using the above stated augmented Lagrange algorithm implemented with FEniCS \cite{LoggMardalEtAl2012a} using the DOLFIN \cite{ LoggWells2010a} Python interface. The arising sub-problems \eqref{prob_auglag} have been solved combining two methods. The first method is the active set method presented by Stadler \cite{Stadler2009}, where optimal control problems of type \eqref{prob_auglag} have been solved, but without augmented state constraints. The second is the method established by Ito et al. \cite{ito2003semi-smooth} that presented an active set method for optimal control problems with state constraints but without a $L^1$-cost term.\\

\noindent Like in \cite{Stadler2009}  we set for \eqref{eq_kkt_alg}
$$\bar\xi_k = \lambda_k^\alpha -\lambda^a_k + \lambda^b_k,$$
where $\lambda_k^\alpha$ denotes the subdifferential of $\beta\norm{\bu_k}_{L^1(\Omega)}$, $\lambda_k^a$ the multiplier to the lower control constraints $u_a-\bu_k\leq 0$ and $\lambda_k^b$ the multiplier corresponding to the upper control constraint $\bu_k-u_b\leq 0$.
Then \eqref{eq:kkt_auglag:3} can be written as
$$ \bp_k + \alpha \bu_k+\bar\xi_k = 0.$$
Defining the following active sets, see also Lemma \ref{lem:opt_rel_a_not0}:
\begin{align*}
\mathcal{Y}_-^k &=  \lbrace x\in \Omega:\  (\mu+\rho(y_{k+1}-\psi))>0\rbrace,\\
\mathcal{Y}_+^k &=  \Omega\setminus \mathcal{Y}_-^k ,\\
\mathcal{A}_a^k &=  \lbrace x\in \Omega:\  p_k > \beta - \alpha u_a \rbrace,\\
\mathcal{A}_0^k &=  \lbrace x\in \Omega:\  |p_k|<\beta\rbrace,\\
\mathcal{A}_b^k &=  \lbrace x\in \Omega:\  p_k < -\alpha u_b - \beta\rbrace,\\
\mathcal{I}_-^k &=  \lbrace x\in \Omega:\  \beta \leq p_k \leq \beta - \alpha u_a\rbrace,\\
\mathcal{I}_+^k &=  \lbrace x\in \Omega:\  - \alpha u_b - \beta \leq p_k \leq -\beta \rbrace,\\
\end{align*}
The resulting sub-problem of the augmented Lagrange method can now be solved by the following algorithm.
\begin{algorithm}[H]
\textbf{Active set method (inner iteration)} 
\begin{enumerate}
\item Choose initial data $u_0, p_0$ and parameters $\alpha, \rho$, compute the active sets $\mathcal{Y}_-^0$, $\mathcal{Y}_+^0$, $\mathcal{A}_a^0$, $\mathcal{A}_0^0$, $\mathcal{A}_b^0$, $\mathcal{I}_-^0, \mathcal{I}_+^0$ and set $k := 0$.
\item Solve for $(u_{k+1},y_{k+1}, p_{k+1}, \xi_{k+1})$ satisfying
\begin{align}\label{eq:kktactset}
\begin{split}
Ay_{k+1} -u_{k+1} -f  &=0 ,\\
-A^*p_{k+1} + y_{k+1}-y_d + \mu_{k+1} &= 0,\\
p_{k+1}+\alpha u_{k+1} +\xi_{k+1} &= 0,
\end{split}
\end{align}
\begin{align}\label{eq:funcOnActSet}
\begin{split}
u_{k+1} = \begin{cases}
u_a \quad &\text{on } \mathcal{A}_a^k,\\
0 \quad &\text{on } \mathcal{A}_0^k,\\
u_b \quad &\text{on } \mathcal{A}_b^k,
\end{cases}
\qquad
\xi_{k+1}=\begin{cases}
-\beta\quad &\text{on } \mathcal{I}_-^k,\\
\beta \quad &\text{on } \mathcal{I}_+^k,\\
\end{cases}
\qquad
\mu_{k+1} &=
\begin{cases}
0 \quad &\text{on } \mathcal{Y}_-^k,\\
 \mu+\rho(y_{k+1}-\psi) \quad &\text{on } \mathcal{Y}_+^k.\\
\end{cases}
\end{split}
\end{align}
\item Compute the active sets $\mathcal{Y}_-^{k+1}, \mathcal{Y}_+^{k+1}, \mathcal{A}_a^{k+1}, \mathcal{A}_0^{k+1}, \mathcal{A}_b^{k+1}, \mathcal{I}_-^{k+1}, \mathcal{I}_+^{k+1}$
\item If the following equalities hold: $\mathcal{A}_a^{k+1} =\mathcal{A}_a^k$, $\mathcal{A}_0^{k+1} =\mathcal{A}_0^k$, $\mathcal{A}_b^{k+1} =\mathcal{A}_b^k$, $\mathcal{I}_-^{k+1}=\mathcal{I}_-^k$, $\mathcal{I}_+^{k+1}=\mathcal{I}_+^k$, $\mathcal{Y}_-^{k+1}=\mathcal{Y}_-^k$ and $\mathcal{Y}_+^{k+1}=\mathcal{Y}_+^k$ then go step 5. Otherwise set $k = k+1$ and go to step 2.
\item Compute the subdifferential $\lambda_{k+1} := P_{[-1,1]}\left(-\frac{1}{\beta}\xi_{k+1}\right)$.
\end{enumerate}
\end{algorithm}
\noindent The computation of the $L^1$-subdifferential follows from a projection formular similar to the one from Lemma \ref{lem:opt_rel_a_not0}. Since the active sets are disjoint subsets of $\Omega$ the calculation of $\xi_{k+1}$ in Step 4 does not evoke any conflicts to its usage on the subsets $\mathcal{I}_-^k,\mathcal{I}_+^k$ in Step 3 of the algorithm. Further, the termination criterion yields a solution of the augmented Lagrange subproblem \eqref{eq:RegProb}.
\begin{lemma}
If $\mathcal{A}_a^{k+1} =\mathcal{A}_a^k, \mathcal{A}_0^{k+1} =\mathcal{A}_0^k, \mathcal{A}_b^{k+1} =\mathcal{A}_b^k, \mathcal{I}_-^{k+1}=\mathcal{I}_-^k,\mathcal{I}_+^{k+1}=\mathcal{I}_+^k,\mathcal{Y}_-^{k+1}=\mathcal{Y}_-^k,\mathcal{Y}_+^{k+1}=\mathcal{Y}_+^k$, then $(u_{k},y_{k}, p_{k},\mu_k,\lambda_k)$ is the solution $(\bu_{k},\by_{k}, \bp_{k},\bar\mu_k,\bar\lambda_k)$ to \eqref{eq_kkt_alg} with $\alpha,\mu,\beta$ fixed.
\end{lemma}
\begin{proof}
Since for given active sets the solution to \eqref{eq:kktactset} is unique we have $(u_{k+1},y_{k+1}, p_{k+1}) = (u_{k},y_{k}, p_{k})$. By definition of the active sets $\mathcal{Y}_-^k,\mathcal{Y}_+^k$  we get
$\mu_{k+1} = ( \mu+\rho(y_{k+1}-\psi))_+$. The optimality condition \eqref{eq_kkt_alg_3} can be equivalently expressed by
\begin{align*}
\bu_k&-\max(0,\bu_k+c(\bar\xi_k -\beta)) -\min(0,\bu_k+c(\bar\xi_k +\beta))\\
&+\max(0,(\bu_k-u_b)+c(\bar\xi_k -\beta)) +\min(0,(\bu_k-u_a)+c(\bar\xi_k +\beta)) =0.
\end{align*}
where $c>0$ arbitrary.
Choosing $c=\alpha^{-1}$ and exploiting $\xi_\alpha^\rho = -(p_\rho^\alpha+\alpha \bu_k)$ we get
\begin{align*}
\bu_k&-\alpha^{-1}\max(0,-\bp_k -\beta) -\alpha^{-1}\min(0,-\bp_k +\beta)\\
&+\alpha^{-1}\max(0,-\bp_k -\beta-\alpha u_b) +\alpha^{-1}\min(0,-\bp_k +\beta-\alpha u_a) =0,
\end{align*}
which is satisfied for $\bu_k=u_{k+1},\bp_k=p_{k+1}$ defined by \eqref{eq:funcOnActSet}. Moreover, $\lambda_{k+1}$ satisfies \eqref{eq_kkt_alg_2} by definition.
Consequently $(u_{k},y_{k}, p_{k},\mu_k,\lambda_k)$ satisfies \eqref{eq_kkt_alg}.
\end{proof}
\noindent However, high values of the penalty parameter $\rho$ paired with small values of the Tikhonov parameter $\alpha$ may evoke bad stability during solution of the subproblem. To counteract this aspect we introduce a so called intermediate step. Here, Step 3 and Step 4 of Algorithm \ref{alg_detail} are extended for a third alternative. If the current iterates of the $k$-th iteration do not satisfy the update rule but sufficiently satisfy the feasibility and complementarity condition, i.e.
$$ R_k \geq \tau R_{n-1}^+ \quad \text{and}\quad  \norm{(\by_k-\psi)_+}_{\CO} + |(\bmu_k,\psi -\by_k)| < \eps_I,$$
with $\eps_I>0$, we set
\begin{align*}
\alpha_{k+1} &= \omega \alpha_k,\\
\mu_{k+1} &= \bmu_k, \\
(u_n^+,y_n^+,p_n^+) :&= (\bu_k, \by_k, \bp_k).
\end{align*}

\noindent As a termination criterion we check the optimality conditions of the current iterate $(u_n^+,y_n^+,p_n^+,\mu_n^+,\lambda_n^+)$ i.e. we stop the algorithm if the inequality
\begin{equation}\label{eq:breaking_condition}
\begin{split}
&\norm{u_n^+(x)-P_{[u_a(x), u_b(x)]} \big(u_n^+(x) -(p_n^+ + \beta \lambda_n^+(x)) \big)}_{L^2(\Omega)} \\
&\qquad + \norm{(y_n^+(x)-\psi)_+}_{\CO} + |(\mu_n^+(x),y_n^+(x)-\psi)| \leq \eps
\end{split}
\end{equation}
is satisfied. In order to be consistent we set $\eps_I < \eps$.\\

\noindent As the Active-Set methods are related to the class of semi-smooth Newton methods we cannot expect a global convergence behavior of the method described above. Furthermore, the problem becomes bad conditioned if $\alpha \to 0$ or $\rho \to \infty$. Due to the intermediate step we expect $\rho$ to be bounded. However as $\alpha$ goes to zero we have to globalize our method. We use a projected gradient method to construct suitable starting values for the Active-Set method.

\section{Numerical Results}\label{sec:numerics}
Let us present some numerical results to support our method. We apply our method for problems of the following form:
\begin{gather}
\min\ J(y,u):=\frac{1}{2}||y-y_d||_{L^2(\Omega)}^2 + \beta \|u\|_{L^1(\Omega)}\\
\intertext{subject to}
\begin{alignedat}{2}
Ay &=u+f &\quad& \text{ in }\Omega,\\
y&= 0 && \text{ on }\partial\Omega,\\
        y &\leq \psi &&\text{ in } \Omega,\\
        u_a\leq u &\leq u_b && \text{ in } \Omega.
\end{alignedat}
\notag
\end{gather}
The additional variable $f \in L^2(\Omega)$ allows us to construct test problems with known solutions.

\subsection{Example 1: Bang-Bang-Off Example in One Space Dimension}
We first consider the one-dimensional case and define $\Omega = (-1,1)$, $u_a = -1$, $u_b = 1$ and $\beta = 1$. Furthermore set
\begin{align*}
\bar y(x) &:= \begin{cases} 28 + 108 \cdot x + 144 \cdot x^2 + 64\cdot x^3 &\text{if} \quad x \in [-1, - \frac{3}{4}]\\ 1 &\text{if}\quad x \in [ -\frac{3}{4}, \frac{3}{4} ]\\28 - 108 \cdot x + 144 \cdot x^2 - 64\cdot x^3 &\text{if} \quad x \in [ \frac{3}{4}, 1]  \end{cases}\\
\bar p(x) &:=-2 \cos \left( \frac{3 \pi}{2} x \right)\\
\bar u(x) &:= \begin{cases}0 &\text{if}\quad x \in [-1, - \frac{8}{9}]\cup [-\frac{4}{9}, - \frac{2}{9}] \cup [\frac{2}{9}, \frac{4}{9}] \cup [\frac{8}{9},1]\\
1 &\text{if}\quad  x \in (-\frac{2}{9}, \frac{2}{9})\\
-1 &\text{if}\quad x \in (-\frac{8}{9}, - \frac{4}{9})\cup (\frac{4}{9}, \frac{8}{9}) \end{cases}\\
\bar \mu(x) &:=   \begin{cases} \text{Exp} \left( - \frac{1}{1-\left(\frac{4}{3}x \right)^2} \right)  & \text{if } x \in [-\frac{3}{4}, \frac{3}{4}]\\ 0 & \text{else} \end{cases}\\
\psi(x) &:= 1.
\end{align*}
Some calculations show that $\bar y, \bar p \in C^2(\Omega)$ and $\bar y  = \bar p = 0$ on $\partial \Omega$. By construction we obtain $\bar u(x) \in \{-1,0, 1\}$ for a.e. $x \in \Omega$. In order to satisfy the optimality conditions we now set
\begin{align*}
f(x) &:= - \Delta \bar y(x) - \bar u(x),\\
y_d(x) &:= \Delta \bar p(x) + \bar y(x) + \bar \mu(x).
\end{align*}
One now can check that the functions $(\bar u, \bar y, \bar p, \bar \mu)$ satisfy the KKT conditions defined in Theorem \ref{theo:ex-adjoint-multiplier} with a suitable modification for the forward equation. We apply our algorithm with the following set of parameters
$$\theta = 5,\quad \omega = 0.75,\quad \tau = 0.8,\quad \eps = 10^{-6},\quad \eps_I = 5 \cdot 10^{-7}.$$
The interval $\Omega$ is divided into $10^6$ equidistant elements. The algorithm stops after a total of $40$ iterations, which splits in $13$ successful, $19$ intermediate and $8$ not successful iterations with an average of $5.25$ inner iterations. The parameters were initialized with $\alpha := 1$ and $\rho := 100$ and the final parameters are $\alpha = 0.75^{32} \approx 10^{-4}$ and $\rho = 100 \cdot 5^8 \approx 3.9 \cdot 10^7$.\\
As we have an exact solution we can compute convergence rates. We plot the $L^2$-error $\|u_k^+ - \bu\|_{L^2(\Omega)}$ over the regularization parameter $\alpha_k$. Note the we only plot successful and intermediate steps. As expected we see that the algorithm produces only intermediate steps after some given time. The error can be found in Figure \ref{fig:error_plot_bbo} and plots of the computed solution can be seen in Figure \ref{fig:1d_bbo_1} and \ref{fig:1d_bbo_2}.

\begin{figure}[htbp]
\makebox[\textwidth]{\begin{tikzpicture}
\begin{loglogaxis}[
    title=Error Plot Bang-Bang-Off Example,
    xlabel={regularization parameter $\alpha_k$},
    ylabel=error $\|u_k^+ - \bar u\|_{L^2(\Omega)}$\\\\,
    ylabel style={align=center}, 
    xminorticks=true,
    xmin=0,
    xmax=1,
    height=8cm,
    grid=major,
    legend entries={successful step, intermediate step},
    legend pos=north west,
]



\addplot[color=green, mark=square*, only marks] plot coordinates {

(0.75,0.513471048169)
(0.5625,0.428920400657)
(0.421875,0.362742234977)
(0.31640625,0.307655661246)
(0.2373046875,0.254170066002)
(0.177978515625,0.213287848386)
(0.133483886719,0.177565282325)
(0.100112915039,0.146071409452)
(0.0750846862793,0.118333302969)
(0.0563135147095,0.0941693625061)
(0.0422351360321,0.0741063495883)
(0.0316763520241,0.058214173627)
(0.0100225957576,0.0285062129442)

    };

\addplot[color=red, mark=triangle*, only marks] plot coordinates {

(0.0237572640181,0.0455875999622)
(0.0178179480135,0.0379012173953)
(0.0133634610102,0.0328757714588)
(0.00751694681821,0.0247105367028)
(0.00563771011366,0.0214173830645)
(0.00422828258525,0.0185632542403)
(0.00317121193893,0.0160895879032)
(0.0023784089542,0.0139488459513)
(0.00178380671565,0.012100190261)
(0.00133785503674,0.010508540984)
(0.00100339127755,0.00914635589823)
(0.000752543458165,0.00800295916376)
(0.000564407593624,0.00704546460916)
(0.000423305695218,0.00627259216674)
(0.000317479271413,0.00568711386704)
(0.00023810945356,0.00530462552501)
(0.00017858209017,0.00509385114437)
(0.000133936567628,0.00493277880199)
(0.000100452425721,0.0048530525714)

    };

\addplot[color=black, mark=triangle*, no markers] plot coordinates {

(0.75,0.513471048169)
(0.5625,0.428920400657)
(0.421875,0.362742234977)
(0.31640625,0.307655661246)
(0.2373046875,0.254170066002)
(0.177978515625,0.213287848386)
(0.133483886719,0.177565282325)
(0.100112915039,0.146071409452)
(0.0750846862793,0.118333302969)
(0.0563135147095,0.0941693625061)
(0.0422351360321,0.0741063495883)
(0.0316763520241,0.058214173627)
(0.0237572640181,0.0455875999622)
(0.0178179480135,0.0379012173953)
(0.0133634610102,0.0328757714588)
(0.0100225957576,0.0285062129442)
(0.00751694681821,0.0247105367028)
(0.00563771011366,0.0214173830645)
(0.00422828258525,0.0185632542403)
(0.00317121193893,0.0160895879032)
(0.0023784089542,0.0139488459513)
(0.00178380671565,0.012100190261)
(0.00133785503674,0.010508540984)
(0.00100339127755,0.00914635589823)
(0.000752543458165,0.00800295916376)
(0.000564407593624,0.00704546460916)
(0.000423305695218,0.00627259216674)
(0.000317479271413,0.00568711386704)
(0.00023810945356,0.00530462552501)
(0.00017858209017,0.00509385114437)
(0.000133936567628,0.00493277880199)
(0.000100452425721,0.0048530525714)

};

\end{loglogaxis}
\end{tikzpicture} }
\caption{Error $\|u_k^+ - \bar u\|_{L^2(\Omega)}$ over $\alpha_k$ for example 1.}
\label{fig:error_plot_bbo}
\end{figure}
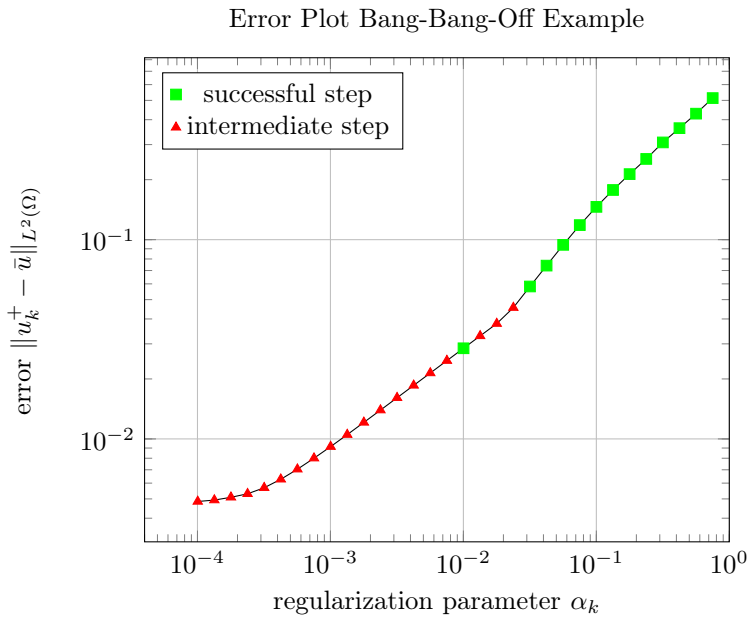

\begin{figure}[htbp]
\makebox[\textwidth]{\includegraphics[width=7cm]{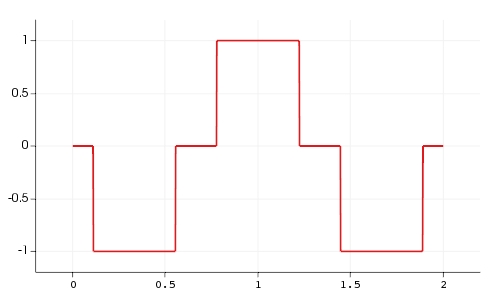} \includegraphics[width=7cm]{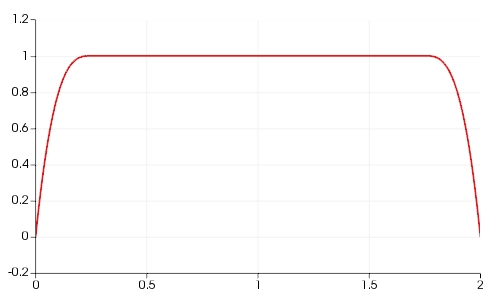}  }
\caption{Computed control $u$ and state $y$ for example 1.}
\label{fig:1d_bbo_1}
\end{figure}

\begin{figure}[htbp]
\makebox[\textwidth]{\includegraphics[width=7cm]{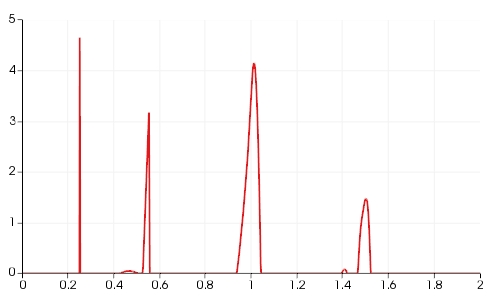} \includegraphics[width=7cm]{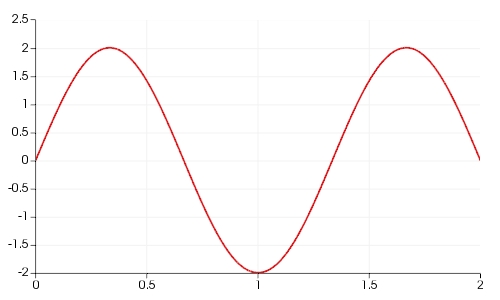}  }
\caption{Computed multiplier $\mu$ and adjoint state $p$ for example 1.}
\label{fig:1d_bbo_2}
\end{figure}

\begin{remark}
Analysing the error $\|u_k^+ - \bu\|_{L^2(\Omega)}$ we see that the error behaves like
\begin{equation}\label{eq:error_estimate}
\|u_k^+ - \bu\|_{L^2(\Omega)} = \mathcal{O}\left( \alpha^\gamma \right)
\end{equation}
with constant $\alpha, \gamma > 0$. We want to mention that the exact control $\bar u$ satisfies the following regularity assumption
$$\text{meas}\{ x\in \Omega: \; \big |  |\bar p(x)| - \beta \big| < \eps \} \leq c \eps^\kappa,$$
for all $\eps > 0$ with some $\kappa > 0$, which can be used to prove error estimates of the form \eqref{eq:error_estimate} for some algorithms, see e.g. \cite{wachsmuth2011b,wachsmuth2016}. However it is an open problem to prove convergence rates for the augmented Lagrange method presented in this paper.
\end{remark}

\subsection{Example 2: Bang-Bang-Off Example in Two Space Dimension}
We set $u_a = -1$, $u_b=1$. Let $\Omega$ be the circle around $0$ with radius $2$. We now define the following functions. For clarity and to shorten our notation we set $r := r(x,y) := \sqrt{x^2 + y^2}$.
\begin{align*}
\bar y(x,y) &:= \begin{cases} 1 & \text{if } r < 1\\ 32 - 120 \cdot r + 180 \cdot r^2 - 130 \cdot r^3 + 45 \cdot r^4 - 6 \cdot r^5 & \text{if } r \geq 1 \end{cases}\\
\bar p(x,y) &:= \sin(x) \cdot \sin(y) \cdot \left( 1 - \frac{5}{4}r^3 + \frac{15}{16} r^4 - \frac{3}{16} r^5 \right)\\
\bar u(x,y) &:= -\text{Sign}( p(x,y) )\\
\bar \mu(x,y) &:=   \begin{cases} \text{Exp} \left( - \frac{1}{1-r^2} \right)  & \text{if } r < 1\\ 0 & \text{if } r \geq 1 \end{cases}\\
\psi(x,y) &:= 1
\end{align*}
Some calculation show that $\bar \mu, \bar p \in C^2(\bar \Omega)$ and $\bar \mu \in C(\bar \Omega)$. Furthermore $\bar y = \bar p = 0$ on $\partial \Omega$. We now set
\begin{align*}
f(x,y) &:= - \Delta \bar y(x,y) - \bar u(x,y),\\
y_d(x,y) &:= \Delta \bar p(x,y) + \bar y(x,y) + \bar \mu(x,y).
\end{align*}
One now can check that for $\beta = 0$ the functions $(\bar u, \bar y, \bar p, \bar \mu)$ satisfy the KKT conditions defined in Theorem \ref{theo:ex-adjoint-multiplier} leading to a bang-bang solution. For $\beta \neq 0$ we expect the optimal solution to exhibit a bang-bang-off structure. Here no exact solution is known. We computed this problem for different values of $\beta$ on a regular triangular grid with approximately $1.8 \cdot 10^5$ degrees of freedom. The parameter used for this computation are $\tau = 0.8$, $\omega = 0.75$, $\theta = 5$, $\eps = 10^{-6}$ and $\eps_I = 5 \cdot 10^{-7}$. We started with $\alpha = 0.1$ and $\rho = 100$. Additional information for the calculations can be found in Table \ref{tab:data_bbo} while the computed controls can be seen in Figure \ref{fig:bang_bang_off_diff_beta}.\\

  \begin{table}[H]\renewcommand{\arraystretch}{1.5}
 \begin{tabular}{|c|c|c|>{\Centering}p{2cm}|>{\Centering}p{2cm}|>{\Centering}p{2cm}|>{\Centering}p{2.5cm}|}
  \hline$\beta$ & final $\alpha$ & final $\rho$ & successful steps & intermediate steps & not successful steps & average inner iterations\\\hline
  $0.05$ & $2.38 \cdot 10^{-5}$ & $10^9$ & 15 & 14 & 7 & 2.9 \\\hline
  $0.1$ & $2.38 \cdot 10^{-5}$ & $10^9$ & 16 & 13 & 7 & 3.1 \\\hline
  $0.2$ & $3.17\cdot 10^{-5}$ & $10^{9}$ & 18 & 10 & 7 & 3.0\\\hline
  $1$ & $5.6 \cdot 10^{-5}$ & $10^{10}$ & 20 & 6 & 8 & 3.5\\\hline
 \end{tabular}
 \caption{Additional information for the computation of example 2 for different $\beta$.}
 \label{tab:data_bbo}
 \end{table}


\begin{figure}[H]
\begin{minipage}[b]{0.45\textwidth}
\includegraphics[width=\linewidth]{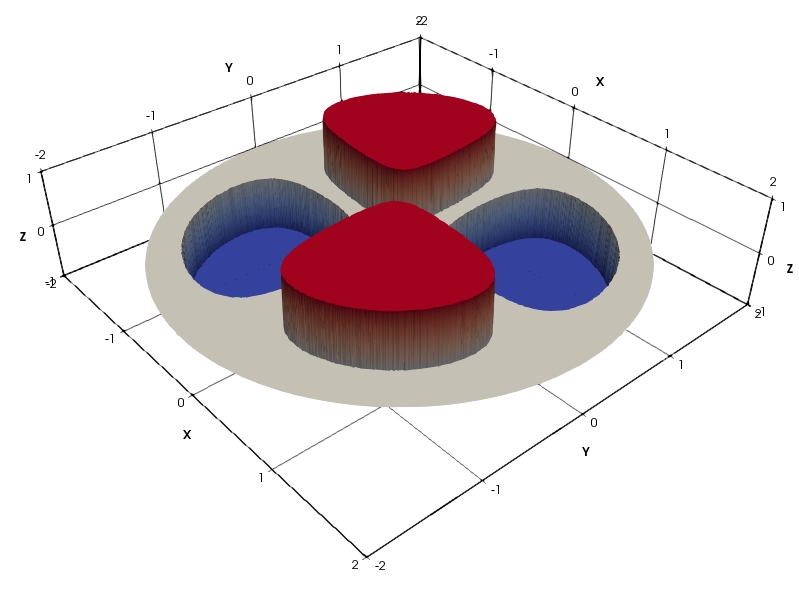}
\end{minipage}
\hfill
\begin{minipage}[b]{0.45\textwidth}
\includegraphics[width=\linewidth]{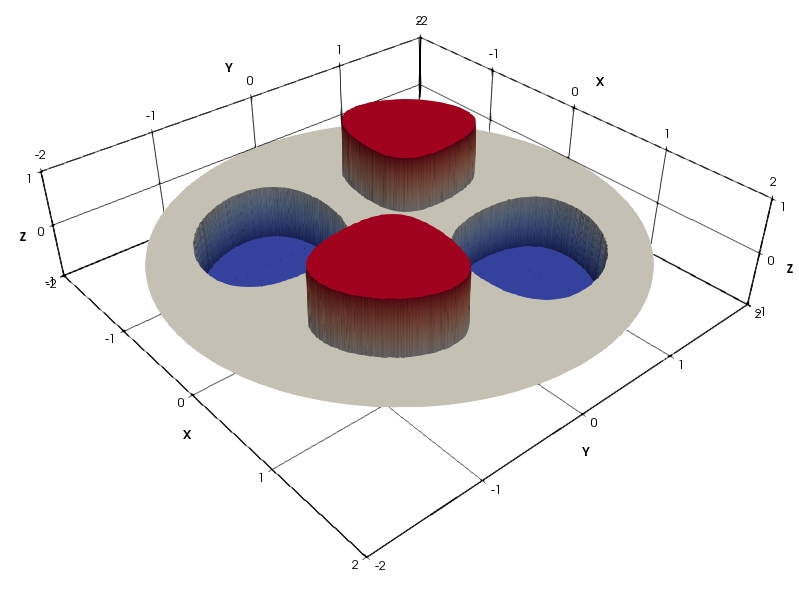}
\end{minipage}
\begin{minipage}[b]{0.45\textwidth}
\includegraphics[width=\linewidth]{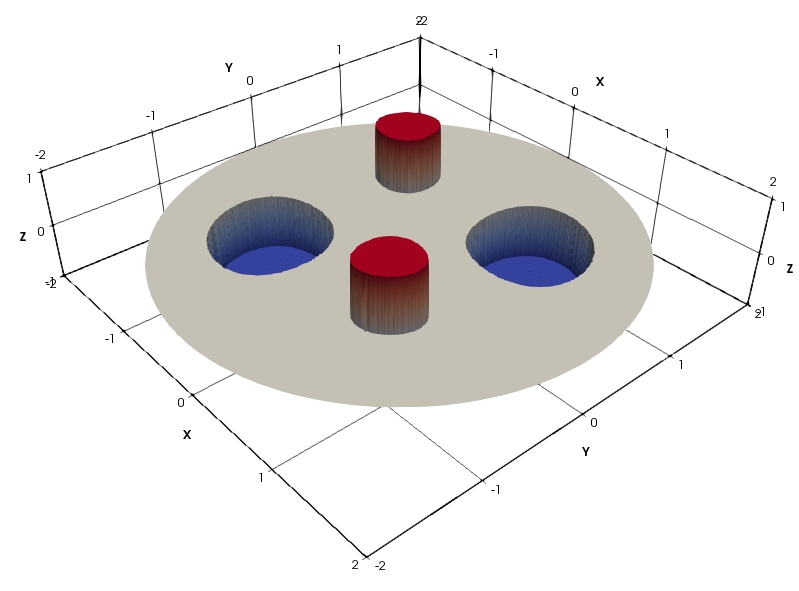}
\end{minipage}
\hfill
\begin{minipage}[b]{0.45\textwidth}
\includegraphics[width=\linewidth]{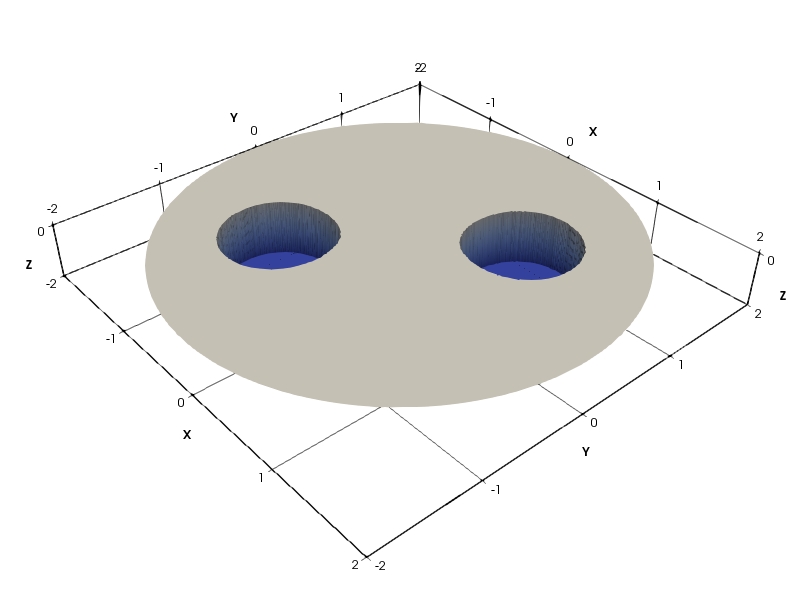}
\end{minipage}
\caption{Computed discrete control for example 2 for different values of $\beta$. From left to right and from top to bottom: $\beta = 0.05$, $\beta = 0.1$, $\beta = 0.2$, $\beta = 1$.}
\label{fig:bang_bang_off_diff_beta}
\end{figure}


\subsection{Example 3}
For the next example we set $\Omega = (0,1)^2$, $u_a = -1$, $u_b = 1$ and $\beta = 10^{-3}$. Furthermore $\tau = 0.8$, $\omega = 0.75$ and $\theta = 10$. Now define
\begin{align*}
\psi(x,y) &:= 0.01\\
y_d(x,y) &:= \frac{1}{2\pi}\sin(\pi x) \sin(\pi y)
\end{align*}
Note that here no exact solution is available. If the state constraints are neglected the exact solution is given by
\begin{align*}
\bar y(x,y) &:= y_d(x,y)\\
\bar u(x,y) &:= \Delta y_d(x,y).
\end{align*}
This example is taken from \cite{poerner2016b} and is an example of an optimal control problem where the desired state is reachable and the source condition $\bar u = S^\ast w$ with an element $w \in L^2(\Omega)$ is satisfied if the state constraints are not present. We computed the solution on a regular triangular grid with $1.6 \cdot 10^5$ degrees of freedom, $\eps = 10^{-6}$ and $\eps_I = 5 \cdot 10^{-7}$. As starting values we set $\alpha = 0.1$ and $\rho = 100$. The algorithm stopped after 8 successful, 25 intermediate and 9 not successful steps with the final values $\alpha = 0.1 \cdot 0.75^{33} \approx 7.5 \cdot 10^{-6}$ and $\rho = 100 \cdot 5^9 \approx 2.0 \cdot 10^{8}$. The computed results can be seen in Figure \ref{fig:donut1} and Figure \ref{fig:donut2} .
\begin{figure}[H]
\begin{minipage}[b]{0.47\textwidth}
\includegraphics[width=\linewidth]{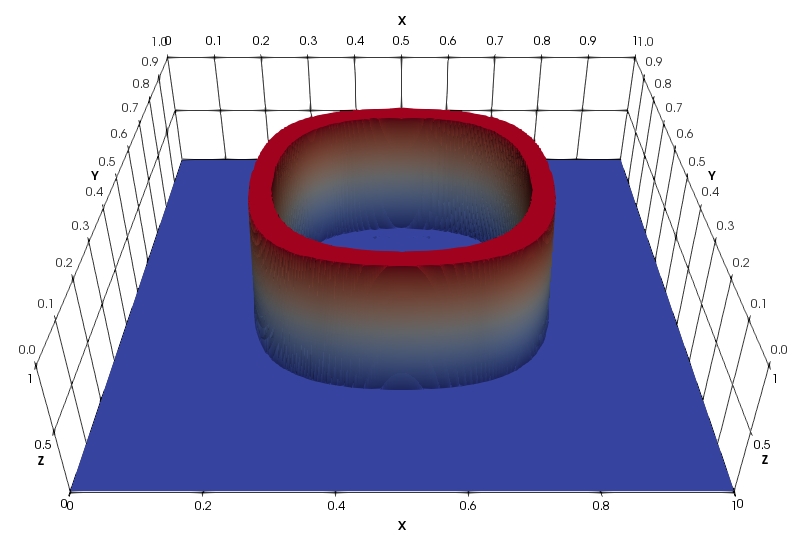}
\end{minipage}
\hfill
\begin{minipage}[b]{0.47\textwidth}
\includegraphics[width=\linewidth]{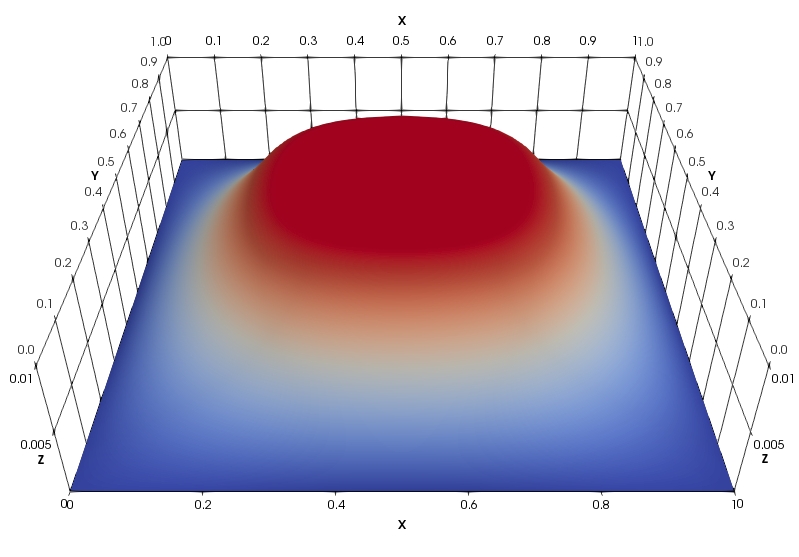}
\end{minipage}
\caption{Computed results for example 3. From left to right: Control $u$, state $y$.}
\label{fig:donut1}
\end{figure}
\begin{figure}[H]
\begin{minipage}[b]{0.45\textwidth}
\includegraphics[width=\linewidth]{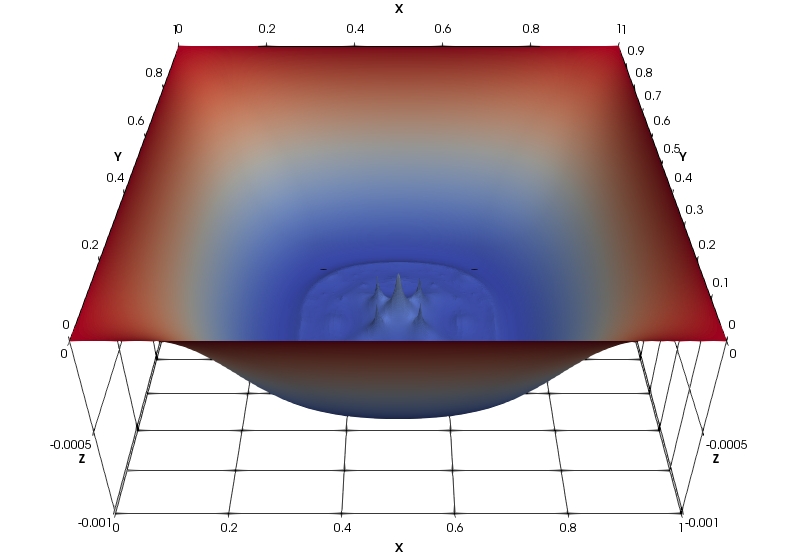}
\end{minipage}
\hfill
\begin{minipage}[b]{0.45\textwidth}
\includegraphics[width=\linewidth]{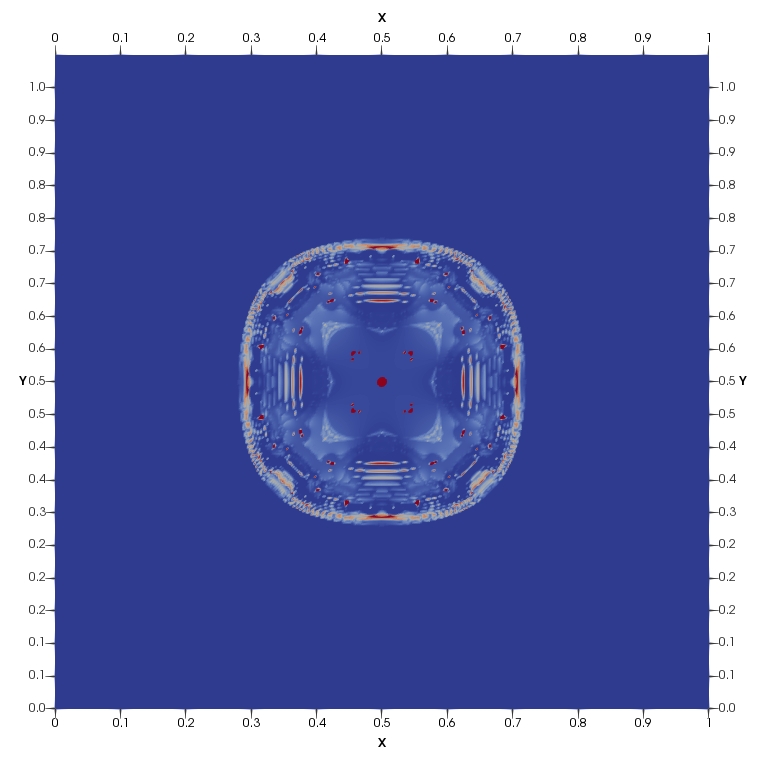}
\end{minipage}
\caption{Computed results for example 3. From left to right: Adjoint state $p$ and multiplier $\mu$. The range of $\mu$ is given by $\mu(x) \in [0,40]$}
\label{fig:donut2}
\end{figure}

\bibliography{mybib}
\bibliographystyle{plain_abbrv}


\end{document}